\documentclass[a4paper]{article}

\usepackage[bookmarks=true]{hyperref}

\usepackage{amsmath}
\usepackage{amssymb}
\usepackage{amsthm}
\usepackage{nicefrac}

\newtheorem{theorem}{Theorem}[section]
\newtheorem{lemma}[theorem]{Lemma}
\newtheorem{proposition}[theorem]{Proposition}
\newtheorem{corollary}[theorem]{Corollary}

\theoremstyle{definition}
\newtheorem{definition}[theorem]{Definition}
\newtheorem{remark}[theorem]{Remark}

\newcommand{\Frechet}{Fr\'{e}chet }

\newcommand{\R}{\mathbb{R}}
\newcommand{\N}{\mathbb{N}}

\newcommand{\Borel}[1]{\mathcal{B}(#1)}
\newcommand{\LLeb}[4]{\mathcal{L}^{#1}([#2,#3],#4)}
\newcommand{\Leb}[4]{L^{#1}([#2,#3],#4)}
\newcommand{\ACp}[3]{AC_{L^p}([#1,#2],#3)}
\newcommand{\abpart}{a=t_0<t_1<\ldots<t_n=b}

\DeclareMathOperator{\esssup}{ess\: sup}
\DeclareMathOperator{\ev}{ev}
\DeclareMathOperator{\id}{id}
\DeclareMathOperator{\Evol}{Evol}
\DeclareMathOperator{\evol}{evol}
\DeclareMathOperator{\pr}{pr}
\DeclareMathOperator{\im}{im}
\DeclareMathOperator{\incl}{incl}
\DeclareMathOperator{\linspan}{span}
\begin{document}

\title{Measurable regularity of infinite-dimensional Lie groups based on Lusin measurability}
\author{Natalie Nikitin}
\date{}
\maketitle

\begin{abstract}
We discuss Lebesgue spaces $\mathcal{L}^p([a,b],E)$ of
Lusin measurable vector-valued functions and the corresponding vector spaces $AC_{L^p}([a,b],E)$
of absolutely continuous functions.
These can be used to construct Lie groups $AC_{L^p}([a,b],G)$ of absolutely
continuous functions with values in an infinite-dimensional Lie group $G$.
We extend the notion of $L^p$-regularity
of infinite-dimensional Lie groups
introduced by Gl\"ockner to this setting and adopt several results and tools.
\end{abstract}

\section*{Introduction}
In \cite{Milnor}, Milnor calls an infinite-dimensional Lie group $G$ (with Lie algebra
$\mathfrak{g}$ and identity element $e$) \emph{regular} if for every smooth curve
$\gamma\colon[0,1]\to \mathfrak{g}$ the initial value problem
	\begin{align}\label{eq:ilp}
		\eta^\prime = \eta.\gamma,\quad \eta(0)=e,
	\end{align}
has a (necessarily unique) solution $\Evol(\gamma)\colon[0,1]\to G$
and the function 
	\begin{align*}
		\evol\colon C^{\infty}([0,1],\mathfrak{g})\to G,\quad \gamma\mapsto\Evol(\gamma)(1)
	\end{align*}
so obtained is smooth.

Further, \cite{GlHReg} and \cite{NeebLieGr} deal with the concept of \emph{$C^k$-regularity},
investigating whether the above initial value problem has a solution for every
$C^k$-curve $\gamma$ (the solution $\Evol(\gamma)$ being a $C^{k+1}$-curve then) and
whether the function $\evol\colon C^k([0,1],\mathfrak{g})\to G$ is smooth.

Generalizing this theory even more, in \cite{GlMR} Gl\"ockner constructs Lebesgue spaces
$L_B^p([a,b],E)$
of Borel measurable functions $\gamma\colon[a,b]\to E$ with 
values in \Frechet spaces $E$ (for $p\in[1,\infty]$) and introduces spaces of
certain absolutely continuous $E$-valued functions $\eta\colon[a,b]\to E$ (denoted by 
$\ACp abE)$ with derivatives in $L_B^p([a,b],E)$. Having a Lie group structure on the
spaces $\ACp 01G$ available, in \cite{GlMR} a Fr\'{e}chet-Lie group $G$ is called 
\emph{$L^p$-semiregular}
if the initial value problem (\ref{eq:ilp}) has a solution $\Evol(\gamma)\in\ACp 01G$ for every
$\gamma\in L_B^p([0,1],\mathfrak{g})$, and $G$ is called \emph{$L^p$-regular} if 
it is $L^p$-semiregular and
the map $\Evol\colon L_B^p([0,1],\mathfrak{g})\to\ACp 01G, \gamma\mapsto\Evol(\gamma)$ is smooth.

Since the sum of two vector-valued Borel measurable functions may be not Borel measurable,
certain assumtions need to be made to obtain a vector space structure on the space of the
considerable maps. This implies that the concepts of \emph{$L^p$-regularity} (mentioned
above) only make sense for Fr\'{e}chet-Lie groups (and some other classes of Lie groups
described in \cite{GlMR}).

To loosen this limitation, we recall the notion of \emph{Lusin-measurable} functions
(Definition \ref{def:lusin-meas}), which have the advantage that vector-valued 
Lusin-measurable functions always form a vector space, and define the corresponding
Lebesgue spaces $\Leb pabE$ in Definition \ref{def:Lebesgue-space-Hausd}. 
Further, in Lemma \ref{lem:lusin-then-borel} and Lemma
\ref{lem:borel-then-lusin}, we recall that under certain conditions there is a close
relation between Lusin and Borel measurable functions (known as the Lusin's Theorem).
This leads to the result that the Lebesgue spaces $L_B^p([a,b],E)$ constructed in \cite{GlMR}
coincide with our Lebesgue spaces $\Leb pabE$, due to the conditions needed for
Borel measurable functions to form a vector space.
(Note that Lebesgue spaces of Lusin measurable functions are also considered in \cite{FlorMPIntegr},
for example.)

We lean on the theory established in \cite{GlMR} and construct $AC_{L^p}$-spaces for sequentially
complete locally convex spaces.
In Definition \ref{def:Lp-regular-Lie-gr} we define the notion of $L^p$-regularity for
infinite-dimensional Lie groups modelled on such spaces 
and adopt several useful results from \cite{GlMR}. In
particular:

\paragraph{Theorem.}
\emph{Let $G$ be an $L^p$-semiregular Lie group. Then 
$\Evol\colon L_B^p([0,1],\mathfrak{g})\to\ACp 01G$ is smooth if and only if $\Evol$
is smooth as a function to $C([0,1],G)$.
}
\medskip

As a consequence, we get:

\paragraph{Theorem.}
\emph{Let $G$ be a Lie group modelled on a sequentially complete locally
convex space and $p,q\in[1,\infty]$ with $q\geq p$. If
$G$ is $L^p$-regular, then $G$ is $L^q$-regular. Furthermore, in this case $G$ is $C^0$-regular.}
\medskip

Furthermore, we show:

\paragraph{Theorem.}
\emph{Let $G$ be a Lie group modelled on a sequentially complete locally convex
space. 
Let $\Omega\subseteq \Leb p01{\mathfrak{g}}$ be an open $0$-neighbourhood.
If for every $\gamma\in\Omega$ there exists the corresponding $\Evol(\gamma)\in\ACp 01G$, 
then $G$ is $L^p$-semiregular.
If, in addition, the function $\Evol\colon\Omega\to\ACp 01G$
is smooth, then $G$ is $L^p$-regular.}

\paragraph{Acknowledgement} The author is deeply grateful to H. Gl\"ockner for numeruous
advice and precious support.
This research was supported by the Deutsche Forschungsgemeinschaft, DFG, 
project number GL 357/9-1. 

\paragraph{Notation}
All the topological spaces are assumed Hausdorff (except for the $\mathcal{L}^p$-spaces),
all vector spaces are $\R$-vector spaces (and locally convex topological vector spaces
are called "locally convex spaces" for short). Wherever we write $[a,b]$, we always mean
an interval in $\R$ with $a<b$.

\section{Measurable functions}

Recall that the Borel $\sigma$-algebra $\Borel X$ on a topological space $X$ is the
$\sigma$-algebra generated by the open subsets of $X$. A function $\gamma\colon X\to Y$
between topological spaces is called \emph{Borel measurable} if the preimage $\gamma^{-1}(A)$
of every open (resp. closed) subset $A\subseteq Y$ is in $\Borel X$. Further, for a measure 
$\mu\colon\Borel X\to [0,\infty]$ on $X$, a subset $N\subseteq X$
is called \emph{$\mu$-negligible} if $N\subseteq N'$ for some $N'\in\Borel X$ with $\mu(N')=0$.
By saying that a property holds for $\mu$-almost every $x\in X$ (or for almost every $x\in X$, or for a.e. $x\in X$)
resp. $\mu$-almost everywhere (or almost everywhere, or a.e.)
we will always mean that there exists some $\mu$-negligible subset $N\subseteq X$ such that
the property holds for every $x\notin N$.

\begin{definition}\label{def:lusin-meas}
Let $X$ be a locally compact topological space, let $Y$ be a topological space and
 let $\mu\colon\Borel X\to [0,\infty]$ be a 
measure on $X$. A function $\gamma\colon X\to Y$ is called
\emph{Lusin $\mu$-measurable} (or \emph{$\mu$-measurable}) if for each compact subset $K\subseteq X$
there exists a sequence $(K_n)_{n\in\N}$ of
compact subsets $K_n\subseteq K$ such that each of the
restrictions $\gamma|_{K_n}$ is continuous and $\mu(K\setminus\bigcup_{n\in\N}K_n)=0$.
\end{definition}

\begin{remark}\label{rem:lusin-cont-comp}
Clearly, every continuous function $\gamma\colon X\to Y$ is $\mu$-measurable, more generally,
if $\gamma|_K$ is continuous for every compact subset $K\subseteq X$, then $\gamma$ is
$\mu$-measurable. Furthermore, it is easy to see that if $f\colon Y\to Z$ is a continuous function
between topological spaces and $\gamma\colon X\to Y$ is  $\mu$-measurable, then
the composition $f\circ\gamma\colon X\to Z$ is $\mu$-measurable. Actually, it may happen that
a composition $\gamma\circ f$ with a continuous function $f\colon X\to X$ is not $\mu$-measurable
(see \cite[IV \S 5 No. 3]{BI1-6}). However, a criterion for a composition of Lusin
measurable functions to be Lusin measurable can be found, for example, 
in \cite[Chapter I, Theorem 9]{Schwartz}. The next lemma describes a special case.
\end{remark}

\begin{lemma}\label{lem:lusin-special-cont-comp}
Let $f\colon X_1\to X_2$ be a continuous function between locally compact spaces,
let $Y$ be a topological space. Let $\mu_1$, $\mu_2$ be measures on $X_1$, $X_2$,
respectively, and assume that $\mu_1(f^{-1}(N))=0$, whenever $\mu_2(N)=0$. Then
for every $\mu_2$-measurable function $\gamma\colon X_2\to Y$, the composition
$\gamma\circ f\colon X_1\to Y$ is $\mu_1$-measurable.
\end{lemma}

\begin{proof}
Given a compact subset $K\subseteq X_1$, the continuous image $L:=f(K)\subseteq X_2$ is
compact, hence there exists a sequence $(L_n)_{n\in\N}$ of compact subsets of $L$
such that $\gamma|_{L_n}$ is continuous for every $n$ and 
$\mu_2(L\setminus\bigcup_{n\in\N}L_n)=0$.
Then each $K_n:=f^{-1}(L_n)\cap K$ is closed in $K$, hence compact and 
$\gamma\circ f|_{K_n}$ is continuous. Furthermore,
	\begin{align*}
		\mu_1(K\setminus\bigcup_{n\in\N} K_n)
		=\mu_1(K\setminus f^{-1}(\bigcup_{n\in\N}L_n))
		\leq \mu_1(f^{-1}(L\setminus\bigcup_{n\in\N}L_n))=0,
	\end{align*}
by assumption on $f$, whence $\gamma\circ f$ is $\mu_1$-measurable.
\end{proof}

We recall that a measure $\mu\colon \Borel X\to[0,\infty]$ on a locally compact space $X$
is called a \emph{Radon measure} if $\mu$ is inner regular and $\mu(K)<\infty$ for every compact
subset $K\subseteq X$. An essential criterion for measurability with respect to Radon measures
is recited below (and can be found, for example, in \cite[IV \S 5 No. 1]{BI1-6}).

\begin{lemma}\label{lem:lusin-meas-iff}
Let $X$ be a locally compact topological space, $Y$ be a topological space and let 
$\mu\colon\Borel X\to [0,\infty]$ be a Radon measure. A function
$\gamma\colon X\to Y$ is $\mu$-measurable if and only if for each $\varepsilon>0$
and each compact subset $K\subseteq X$ there exists a compact subset $K_\varepsilon\subseteq K$
such that $\gamma|_{K_\varepsilon}$ is continuous and $\mu(K\setminus K_\varepsilon)\leq\varepsilon$.
\end{lemma}

\begin{proof}
First, assume that $\gamma$ is $\mu$-measurable, let $K\subseteq X$ be a compact subset 
and fix $\varepsilon>0$. Given a sequence
$(K_n)_{n\in\N}$ of compact sets $K_n\subseteq K$ as in Definition \ref{def:lusin-meas}, the sequence
$(K\setminus \bigcup_{m=1}^n K_m)_{n\in\N}$ is decreasing and $\mu(K\setminus K_1)<\infty$.
Hence, we have
	\begin{align*}
		\lim_{n\to\infty} \mu(K\setminus\bigcup_{m=1}^n K_m)
		=\mu(\bigcap_{n\in\N}(K\setminus\bigcup_{m=1}^n K_m))
		=\mu(K\setminus\bigcup_{n\in\N}K_n)=0,
	\end{align*}
therefore, there exists some $N_\varepsilon\in\N$ such that 
$\mu(K\setminus\bigcup_{m=1}^{N_\varepsilon} K_m)\leq\varepsilon$.
As the union is finite, the set $K_\varepsilon:=\bigcup_{m=1}^{N_\varepsilon} K_m$ is compact and
$\gamma|_{K_\varepsilon}$ is continuous.

Now, we show that the converse is also true. To this end, for a compact set $K\subseteq X$
we construct a sequence $(K_n)_{n\in\N}$
of compact subsets of $K$ such that for each $n$ the restriction $\gamma|_{K_n}$ is continuous  
and $\mu(K\setminus\bigcup_{m=1}^n K_m)\leq\nicefrac{1}{n}$. Then
	\begin{align*}
		\mu(K\setminus\bigcup_{n\in\N}K_n)
		=\mu(\bigcap_{n\in\N}(K\setminus\bigcup_{m=1}^n K_m))
		=\lim_{n\to\infty} \mu(K\setminus\bigcup_{m=1}^n K_m)
		\leq \lim_{n\to\infty}\frac{1}{n}=0,
	\end{align*}
using the same arguments as in the first part of the proof, hence $\gamma$ will be $\mu$-measurable.

The construction of the sequence $(K_n)_{n\in\N}$ is made as follows.
For $n=1$, by assumption there is a compact subset $K_1\subseteq K$ such that $\gamma|_{K_1}$
is continuous and $\mu(K\setminus K_1)\leq1$ (taking $\varepsilon=1$). Having constructed
$K_1,\ldots,K_n$ with the required properties, we choose $\varepsilon=\nicefrac{1}{2(n+1)}$, then
the assumption yields a compact subset $K_\varepsilon\subseteq K$ such that 
$\gamma|_{K_\varepsilon}$ is continuous and $\mu(K\setminus K_\varepsilon)\leq\varepsilon$.
But by the inner regularity of $\mu$, for the measurable set 
$A:=K_\varepsilon\setminus\bigcup_{m=1}^n K_m\subseteq K$ there is a compact subset 
$K_{n+1}\subseteq A$ such that $\mu(A\setminus K_{n+1})\leq\varepsilon$ (since $\mu(A)<\infty$). As 
$K_{n+1}\subseteq K_\varepsilon$, the restriction $\gamma|_{K_{n+1}}$ is continuous
and we have
	\begin{align*}
		\mu(K\setminus\bigcup_{m=1}^{n+1} K_m)
		=\mu(K\setminus K_\varepsilon)+\mu(A\setminus K_{n+1})
		\leq\varepsilon=\frac{1}{n+1}.
	\end{align*}
\end{proof}

\begin{remark}\label{rem:lusin-meas-special-case}
Given a Radon measure $\mu$ on $X$, for a $\mu$-measurable function $\gamma\colon X\to Y$
and every compact subset $K\subseteq X$
we can always construct a sequence $(K_n)_{n\in\N}$ (proceeding as in the second part of the previous 
proof) which is pairwise disjoint,
this will be useful in Subsection \ref{subsec:weak-integrable-functions}, where we
consider some integrals of $\mu$-measurable functions. Using the inner regularity of $\mu$,
such a sequence can also be constructed for every Borel set $B\subseteq X$ with $\mu(B)<\infty$,
so that every function 
$\bar{\gamma}\colon X\to Y$ such that $\bar{\gamma}(x)=\gamma(x)$ for $\mu$-almost
every $x\in X$ is $\mu$-measurable.

Furthermore, if the locally compact space $X$ is $\sigma$-compact (in particular, if $X$ is compact),
then there is a sequence 
$(K_n)_{n\in\N}$ such that $\mu(X\setminus\bigcup_{n\in\N}K_n)=0$.
\end{remark}

In the following, the measure $\mu\colon\Borel X\to [0,\infty]$ on a locally compact
space $X$ will be always assumed a Radon measure. The proof of the following lemma
can be found (partially) in \cite[IV \S 5 No. 1]{BI1-6}

\begin{lemma}\label{lem:lusin-into-products}
Let $X$ be a locally compact topological space, and
$(Y_n)_{n\in\N}$ be topological spaces. A map 
$\gamma:=(\gamma_n)_{n\in\N}\colon X\to\prod_{n\in\N}Y_n$ 
is $\mu$-measurable if and only if each of the components $\gamma_n\colon X\to Y_n$ is 
$\mu$-measurable.
\end{lemma}

\begin{proof}
Assume that $\gamma$ is $\mu$-measurable. Each of the coordinate projections 
$\pr_n\colon\prod_{n\in\N}Y_n\to Y_n$ is continuous, thus
by Remark \ref{rem:lusin-cont-comp} each of the compositions $\pr_n\circ\gamma = \gamma_n$
is $\mu$-measurable.

On the other hand, fix $\varepsilon>0$ and a compact subset $K\subseteq X$. Using Lemma 
\ref{lem:lusin-meas-iff} for each $n\in\N$ we find a compact subset $K_n\subseteq K$ such that
$\gamma_n|_{K_n}$ is continuous and $\mu(K\setminus K_n)\leq\nicefrac{\varepsilon}{2^n}$, as each
$\gamma_n$ is $\mu$-measurable. Then the intersection $K_\varepsilon := \bigcap_{n\in\N} K_n$
is a compact subset of $K$ with
	\begin{align*}
		\mu(K\setminus K_\varepsilon)
		=\mu(\bigcup_{n\in\N} (K\setminus K_n))
		\leq\sum_{n=1}^\infty \mu(K\setminus K_n)
		\leq \sum_{n=1}^\infty \frac{\varepsilon}{2^n}
		=\varepsilon.
	\end{align*}
Since $K_\varepsilon\subseteq K_n$ for each $n\in\N$, the restriction $\gamma|_{K_\varepsilon}$
is continuous, thus $\gamma$ is $\mu$-measurable by Lemma \ref{lem:lusin-meas-iff}.
\end{proof}

\begin{remark}\label{rem:lusin-vector-space}
If $E$ is a topological vector space, then for $\mu$-measurable functions
$\gamma,\eta\colon X\to E$ and every scalar $\alpha\in\R$, the functions
$\gamma+\eta\colon X\to E$, $\alpha\gamma\colon X\to E$ are $\mu$-measurable
(by Remark \ref{rem:lusin-cont-comp} and Lemma \ref{lem:lusin-into-products}). Thus, vector-valued
$\mu$-measurable functions build a vector space without any assumptions on the range $E$
(if $\mu$ is a Radon measure). This allows us to construct locally convex vector spaces
(the \emph{Lebesgue spaces}) in Section \ref{sec:Lebesgue-spaces} consisting
of certain $\mu$-measurable functions  with values
in arbitrary locally convex spaces, expanding the theory established in \cite{GlMR},
where the author deals with classical Borel measurable functions, which do not necessarily
form a vector space without additional assumptions on $E$.
\end{remark}

The relation between $\mu$-measurable functions and Borel measurable functions is
known as the Lusin's Theorem and can be found in several versions in \cite{B92}, \cite{Elst},
and others. Below, we recall the criterions which suffice for our purposes.

\begin{lemma}\label{lem:lusin-then-borel}
Assume that $X$ is a $\sigma$-compact locally compact space and
let $\gamma\colon X\to Y$ be a $\mu$-measurable function. Then there exists a Borel measurable function
$\bar{\gamma}\colon X\to Y$ such that $\bar{\gamma}(x)=\gamma(x)$ for almost every $x\in X$.
\end{lemma}

\begin{proof}
Given a sequence of compact subsets $(K_n)_{n\in\N}$ as in Remark \ref{rem:lusin-meas-special-case}, 
we set $X\setminus\bigcup_{n\in\N} K_n=:N\in\Borel X$ and
define the function $\bar{\gamma}\colon X\to Y$ via
	\begin{align}
		\bar{\gamma}(x):=\gamma(x)\mbox{, if } x\in \bigcup_{n\in\N}K_n,\quad
		\bar{\gamma}(x):=y_0\in Y\mbox{, if } x\in N.
	\end{align}
Then obviously $\bar{\gamma}(x)=\gamma(x)$ a.e. Now we show
that $\bar{\gamma}^{-1}(A)\in\Borel X$ for every closed subset $A\subseteq Y$. Consider
	\begin{align*}
		\bar{\gamma}^{-1}(A) = (\bar{\gamma}^{-1}(A)\cap N)\cup(\bar{\gamma}^{-1}(A)\cap\bigcup_{n\in\N}K_n).
	\end{align*}
If $y_0\in A$, then $\bar{\gamma}^{-1}(A)\cap N = N \in\Borel X$, otherwise 
$\bar{\gamma}^{-1}(A)\cap N=\emptyset\in\Borel X$. Further
	\begin{align*}
		\bar{\gamma}^{-1}(A)\cap\bigcup_{n\in\N}K_n 
		= \bigcup_{n\in\N}(\bar{\gamma}^{-1}(A)\cap K_n)
		= \bigcup_{n\in\N}(\gamma^{-1}(A)\cap K_n)\in\Borel X,
	\end{align*}
being a countable union of sets closed in $X$ (being closed in $K_n$).
\end{proof}

\begin{lemma}\label{lem:borel-then-lusin}
Let $Y$ be a topological space with countable base. Then each Borel measurable function 
$\gamma\colon X\to Y$ is $\mu$-measurable.
\end{lemma}

\begin{proof}
We denote by $(V_n)_{n\in\N}$ the countable base of $Y$. Fix a compact subset $K\subseteq X$ and
$\varepsilon>0$. Consider the sets
	\begin{align*}
		B_n:=K\cap \gamma^{-1}(V_n),\quad C_n:=K\setminus\gamma^{-1}(V_n),
	\end{align*}
which are Borel sets with finite measure. Thus, the inner regularity of $\mu$ yields compact subsets
$K_n\subseteq B_n$, $K'_n\subseteq C_n$ such that 
$\mu(B_n\setminus K_n)$, $\mu(C_n\setminus K'_n)\leq\nicefrac{\varepsilon}{2^{n+1}}$.
The intersection $K_\varepsilon:=\bigcap_{n\in\N}(K_n\cup K'_n)$ is a compact subset of $K$ and 
	\begin{align*}
		\mu(K\setminus K_\varepsilon)
		&=\mu(K\setminus\bigcap_{n\in\N}(K_n\cup K'_n)
		=\mu(\bigcup_{n\in\N} K\setminus(K_n\cup K'_n))\\
		&\leq\sum_{n=1}^\infty \mu(K\setminus(K_n\cup K'_n))
		=\sum_{n=1}^\infty (\mu(B_n\setminus K_n) + \mu(C_n\setminus K'_n))\\
		&\leq\sum_{n=1}^\infty\frac{\varepsilon}{2^n} = \varepsilon.
	\end{align*}
It remains to show the continuity of $\gamma|_{K_\varepsilon}$, then $\gamma$ will be
$\mu$-measurable by Lemma \ref{lem:lusin-meas-iff}. Fix some $x\in K_\varepsilon$ and
let $V_x\subseteq Y$ be a neighborhood of $\gamma(x)$. Then $\gamma(x)\subseteq V_m\subseteq V_x$
for some $m$, thus $x\in K_m$, that is $x\in X\setminus K'_m$ which is an open subset in $X$. 
Then $U_x:= K_\varepsilon\cap(X\setminus K'_m)$ is an open $x$-neighborhood in $K_\varepsilon$
and $U_x\subseteq K_m$, whence $\gamma(U_x)\subseteq V_m\subseteq V_x$. Therefore, the restriction
$\gamma|_{K_\varepsilon}$ is continuous in $x$, hence continuous, as $x$ was arbitrary.
\end{proof}

\section{Lebesgue spaces}\label{sec:Lebesgue-spaces}

\subsection{Definition and basic properties}

First, we recall the definition of Lebesgue spaces of real-valued Borel measurable functions.

\begin{definition}\label{def:Lebesgue-spaces-real}
For $p\in[1,\infty[$, we denote by $\mathcal{L}^p([a,b])$ the vector space of Borel measurable
functions $\gamma\colon[a,b]\to \R$ which are $p$-integrable with respect to the 
Lebesgue-Borel measure $\lambda$
(that is $\int_a^b |\gamma(t)|^p \,dt<\infty$), endowed with the seminorm
	\begin{align*}
		\|\gamma\|_{\mathcal{L}^p} := \left(\int_a^b |\gamma(t)|^p \,dt\right)^\frac{1}{p}.
	\end{align*}
Further, $\mathcal{L}^\infty([a,b])$ denotes the vector space of Borel measurable essentially bounded
functions $\gamma\colon[a,b]\to \R$, endowed with the seminorm
	\begin{align*}
		\|\gamma\|_{\mathcal{L}^\infty}:=\esssup_{t\in[a,b]}|\gamma(t)|.
	\end{align*}

Setting $N_p:=\{\gamma\in\mathcal{L}^p([a,b]) : \|\gamma\|_{\mathcal{L}^p}=0\}$, 
we obtain normed vector spaces $L^p([a,b]) := \mathcal{L}^p([a,b]) / N_p$.
\end{definition}

In \cite{GlMR}, the author defines Lebesgue spaces of Borel measurable functions 
with values in \Frechet spaces as follows.

\begin{definition}\label{def:Lebesgue-space-classical}
Let $E$ be a \Frechet space. For $p\in[1,\infty[$, the space $\mathcal{L}_{B}^p([a,b],E)$ is the
vector space of Borel measurable functions $\gamma\colon [a,b]\to E$ such that $\gamma([a,b])$
is separable and $q\circ\gamma\in\mathcal{L}^p([a,b])$ for each continuous seminorm $q$ on $E$. 
The locally convex topology on $\mathcal{L}_{B}^p([a,b],E)$ is defined by the (countable) family of
seminorms	\begin{align*}
		\|. \|_{\mathcal{L}^p,q}\colon \mathcal{L}_{B}^p([a,b],E)\to[0,\infty[,\quad
		\|\gamma\|_{\mathcal{L}^p,q}:=\|q\circ\gamma\|_{\mathcal{L}^p}
		=\left(\int_a^b q(\gamma(t))^p \,dt\right)^\frac{1}{p}.
	\end{align*}
Further, the vector space $\mathcal{L}_{B}^\infty([a,b],E)$ consists of Borel measurable
functions $\gamma\colon [a,b]\to E$ such that $\gamma([a,b])$ is separable and bounded. The locally
convex topology on $\mathcal{L}_{B}^\infty([a,b],E)$ is defined by the (countable) family
of seminorms
	\begin{align*}
		\|.\|_{\mathcal{L}^\infty,q}\colon \mathcal{L}_{B}^\infty([a,b],E)\to[0,\infty[,\quad
		\|\gamma\|_{\mathcal{L}^\infty,q}:=\|q\circ\gamma\|_{\mathcal{L}^\infty}
		=\esssup_{t\in[a,b]} q(\gamma(t)).
	\end{align*}

Setting $N_p:=\{\gamma\in\mathcal{L}_{B}^p([a,b],E) : \gamma(t)=0 \mbox{ for a.e. } t\in[a,b] \}$
we obtain a Hausdorff locally convex space
	\begin{align*}
		L_B^p([a,b],E) := \mathcal{L}_{B}^p([a,b],E) / N_p,
	\end{align*}
consisting of equivalence classes 
	\begin{align*}
		[\gamma] := \{\bar{\gamma}\in \mathcal{L}_{B}^p([a,b],E) : \bar{\gamma}(t) = 
		\gamma(t)\mbox{ for a.e. }t\in[a,b] \}, 
	\end{align*}
with seminorms
	\begin{align*}
		\|[\gamma]\|_{L^p,q}:=\|\gamma\|_{\mathcal{L}^p,q}.
	\end{align*}
\end{definition}

\begin{remark}\label{rem:Lebesgue-space-FEP}
For locally convex spaces $E$ having the property that every separable closed vector
subspace $S\subseteq E$ can be written as a union $S=\bigcup_{n\in\N} F_n$
of vector subspaces $F_1\subseteq F_2\subseteq \cdots$ which are \Frechet spaces
in the induced topology (called (FEP)-\emph{spaces} in \cite{GlMR}), the spaces $\mathcal{L}_B^p([a,b],E)$
and $L_B^p([a,b],E)$ are constructed in \cite{GlMR} in the same way.
\end{remark}

\begin{definition}\label{def:Lebesgue-space-rc}
If $E$ is an arbitrary locally convex space, then the vector space $\mathcal{L}_{rc}^\infty([a,b],E)$
consists of Borel measurable functions $\gamma\colon [a,b]\to E$ such that $\overline{\gamma([a,b])}$
is compact and metrizable. The seminorms 
$\|\gamma\|_{\mathcal{L}^\infty,q}$ (as in Definition \ref{def:Lebesgue-space-classical})
define the locally convex topology on $\mathcal{L}_{rc}^\infty([a,b],E)$.

Setting $N_{rc}:=\{\gamma\in\mathcal{L}_{rc}^{\infty}([a,b],E) : \gamma(t)=0 \mbox{ for a.e. } t\in[a,b] \}$
we obtain a Hausdorff locally convex space
	\begin{align*}
		L_{rc}^{\infty}([a,b],E) := \mathcal{L}_{rc}^{\infty}([a,b],E) / N_{rc},
	\end{align*}
consisting of equivalence classes 
	\begin{align*}
		[\gamma] := \{\bar{\gamma}\in \mathcal{L}_{rc}^{\infty}([a,b],E) : \bar{\gamma}(t) = 
		\gamma(t)\mbox{ for a.e. }t\in[a,b] \},
	\end{align*}
with seminorms
	\begin{align*}
		\|[\gamma]\|_{L^{\infty},q}:=\|\gamma\|_{\mathcal{L}^{\infty},q}.
	\end{align*}
\end{definition}

\begin{remark}
Note that in \cite{GlMR}, the author constructs all of the above Lebesgue spaces even in a more general form,
consisting of Borel measurable functions $\gamma\colon X\to E$ defined on arbitrary measure
spaces $(X,\Sigma,\mu)$.
\end{remark}

As we do not need additional assumptions on the range space $E$ or on the images
of Lusin measurable functions for them to build a vector space, 
we may construct Lebesgue spaces for arbitrary locally
convex spaces. From now on, the measure will always be the Lebesgue-Borel measure $\lambda$
and we will call Lusin $\lambda$-measurable functions just \emph{measurable}.

\begin{definition}\label{def:Lebesgue-space}
Let $E$ be a locally convex space and $p\in[1,\infty]$. We denote by $\LLeb pabE$ the vector
space of measurable functions $\gamma\colon[a,b]\to E$ such that for each continuous seminorm
$q$ on $E$ we have $q\circ\gamma\in\mathcal{L}^p([a,b])$. We endow each $\LLeb pabE$ with
the locally convex topology defined by the family of seminorms 
	\begin{align*}
		\|. \|_{\mathcal{L}^p,q}\colon \LLeb pabE\to[0,\infty[,\quad
		\|\gamma\|_{\mathcal{L}^p,q}:=\|q\circ\gamma\|_{\mathcal{L}^p}.
	\end{align*}
\end{definition}

The next lemma can be found in \cite{FlorMPIntegr}.

\begin{lemma}\label{lem:lusin-almost-zero}
Let $E$ be a locally convex space, $\gamma\colon[a,b]\to E$ be a measurable function.
Then the following assertions are equivalent:
	\begin{itemize}
		\item[(i)] $\gamma(t)=0$ a.e.,
		\item[(ii)] $\alpha(\gamma(t))=0$ a.e., for each continuous linear functional $\alpha$ on E,
		\item[(iii)] $q(\gamma(t))=0$ a.e., for each continuous seminorm $q$ on $E$.
	\end{itemize}
\end{lemma}

\begin{definition}\label{def:Lebesgue-space-Hausd}
For $p\in[1,\infty]$, from the above lemma follows that
	\begin{align*}
		N_p:= \{\gamma\in\LLeb pabE : \gamma(t)=0\mbox{ a.e.}\} = \overline{\{0\}},
	\end{align*}
thus we obtain Hausdorff locally convex spaces
	\begin{align*}
		\Leb pabE := \LLeb pabE/N_p
	\end{align*}
consisting of equivalence classes
	\begin{align*}
		[\gamma]:=\{\bar{\gamma}\in\LLeb pabE : \bar{\gamma}(t)=\gamma(t)\mbox{ a.e.}\},
	\end{align*}
with seminorms
	\begin{align*}
		\|.\|_{L^p,q}\colon \Leb pabE\to[0,\infty[,\quad
		\|[\gamma]\|_{L^p,q}:=\|\gamma\|_{\mathcal{L}^p,q}.
	\end{align*}
\end{definition}

\begin{remark}\label{rem:Lp-inclusions}
For $1\leq p\leq r\leq\infty$ we have
	\begin{align*}
		C([a,b],E)\subseteq
		\LLeb {\infty}abE \subseteq
		\LLeb rabE \subseteq
		\LLeb pabE \subseteq
		\LLeb 1abE
	\end{align*}
with continuous inclusion maps, as for a continuous seminorm $q$ on $E$ we have
	\begin{align*}
		\|\gamma\|_{\mathcal{L}^p,q}\leq (b-a)^{\frac{1}{p}-\frac{1}{r}}\|\gamma\|_{\mathcal{L}^r,q}.
	\end{align*}
(Here, $C([a,b],E)$ is endowed with the topology of uniform convergence, with continuous seminorms
$\|\gamma\|_{\infty,q}=\|\gamma\|_{\mathcal{L}^\infty,q}$.)
\end{remark}

We show that the Lebesgue spaces $L_B^p([a,b],E)$ constructed in \cite{GlMR} coincide with the Lebesgue spaces $\Leb pabE$ defined in Definition \ref{def:Lebesgue-space-Hausd}.

\begin{proposition}\label{prop:Lp-lusin-classic}
If $E$ is a \Frechet space, then $L_{B}^p([a,b],E)\cong\Leb pabE$ as topological
vector spaces, for each $p\in[1,\infty]$. 
\end{proposition}

\begin{proof}
Let $[\gamma]\in L_{B}^p([a,b],E)$. Then $\gamma\colon[a,b]\to E$ is Borel measurable
and $\im(\gamma)$ is separable. Then the co-restriction
$\gamma|^{\im(\gamma)}\colon[a,b]\to\im(\gamma)$ is Borel measurable, hence measurable by
Lemma \ref{lem:borel-then-lusin} because the range is separable and metrizable, thus has a countable
base. Then $\gamma\colon[a,b]\to E$ is measurable.
Furthermore, $q\circ\gamma\in\mathcal{L}^p([a,b])$ for
each continuous seminorm $q$ on $E$, hence $[\gamma]\in\Leb pabE$.

To show that the linear injective function
	\begin{align*}
		\Phi\colon L_{B}^p([a,b],E)\to\Leb pabE,\quad
		[\gamma] \mapsto [\gamma].
	\end{align*} 
is an isomorphism of topological vector spaces, it suffices to prove surjectivity, as the
continuity of $\Phi$ and $\Phi^{-1}$ will be obvious then.

Let $[\gamma]\in\Leb pabE$. 
If $p\in[1,\infty[$, then consider a partition $[a,b]=N\cup\bigcup_{n\in\N}K_n$ and 
define $\bar{\gamma}\colon[a,b]\to E$ via
	\begin{align*}
		\bar{\gamma}(t):=\gamma(t) \mbox{, if } t\in\bigcup_{n\in\N}K_n, \quad
		\bar{\gamma}(t):=0 \mbox{, if } t\in N.
	\end{align*}
Then $\bar{\gamma}$ is Borel measurable (see Lemma \ref{lem:lusin-then-borel}), 
the image $\im(\bar{\gamma})$ is separable
and $q\circ\bar{\gamma}\in\mathcal{L}^p([a,b])$, thus $\bar{\gamma}\in\mathcal{L}_{B}^p([a,b],E)$.
If $p=\infty$, then there is some nul set $N^\prime\subseteq[a,b]$ such that $\gamma([a,b]\setminus N^\prime)$
is bounded. Consider a partition $[a,b]=N^\prime\cup N\cup\bigcup_{n\in\N}K_n$ and define
$\bar{\gamma}\colon[a,b]\to E$ via
	\begin{align*}
		\bar{\gamma}(t):=\gamma(t) \mbox{, if } t\in\bigcup_{n\in\N}K_n, \quad
		\bar{\gamma}(t):=0 \mbox{, if } t\in N^\prime\cup N.
	\end{align*}
Again, $\bar{\gamma}$ is Borel measurable and the image $\im(\bar{\gamma})$ is separable
and bounded. Thus $\bar{\gamma}\in\mathcal{L}_{B}^{\infty}([a,b],E)$. In any case, 
$\Phi([\bar{\gamma}]) = [\gamma]$, hence $\Phi$ is surjective.
\end{proof}

\begin{remark}
If $E$ is an (FEP)-space, then also $L_{B}^p([a,b],E)\cong\Leb pabE$ as topological
vector spaces. To see this, we only need to show that every $\gamma\in\mathcal{L}_B^p([a,b],E)$
is measurable, the rest of the proof is identical to the above. 

Since $\im(\gamma)$ is separable, the vector subspace $\overline{\linspan{(\im(\gamma))}}$ is separable
and closed, hence there is an ascending sequence $F_1\subseteq F_2\subseteq\cdots$
of vector subspaces such that
	\begin{align*}
		\overline{\linspan{(\im(\gamma))}} = \bigcup_{n\in\N} F_n
	\end{align*}
and each $F_n$ is a separable \Frechet space (see \cite[Lemma 1.39]{GlMR}).
Consider the sets $B_1:=\gamma^{-1}(F_1)$, $B_n:=\gamma^{-1}(F_n\setminus F_{n-1})$ for
$n\geq 2$. Then $[a,b]$ is a disjoint union of $(B_n)_{n\in\N}$, each $B_n\in\Borel {[a,b]}$
and $\gamma|_{B_n}\colon B_n\to F_n$ is Borel measurable, hence measurable by Lemma
\ref{lem:borel-then-lusin}. Therefore, $\gamma\colon [a,b]\to E$ is measurable.
\end{remark}

\begin{remark}
If $E$ is an arbitrary locally convex space, then $L_{rc}^{\infty}([a,b],E) \subseteq \Leb {\infty}abE$.
Again, it suffices to prove that each $\gamma\in\mathcal{L}_{rc}^{\infty}([a,b],E)$ is
measurable. This is true (by Lemma \ref{lem:borel-then-lusin}), since the closure of the
image of $\gamma$ is compact and metrizable, hence has a countable base.
\end{remark}

\subsection{Mappings between Lebesgue spaces}

The following results can be found in \cite[1.34, 1.35]{GlMR}.

\begin{lemma}\label{lem:cont-linear-Lp}
Let $E$, $F$ be locally convex spaces and $f\colon E\to F$ be continuous and
linear. If $\gamma\in\LLeb pabE$ for $p\in[1,\infty]$, then $f\circ\gamma\in\LLeb pabF$
and the map
	\begin{align*}
		\LLeb pabf\colon\LLeb pabE\to\LLeb pabF,\quad \gamma\mapsto f\circ\gamma
	\end{align*}
is continuous and linear.
\end{lemma}

\begin{proof}
From Remark \ref{rem:lusin-cont-comp} follows that $f\circ\gamma$ is measurable.
Further, for every continuous seminorm $q$ on $F$, the composition $q\circ f$
is a continuous seminorm on $E$, whence $q\circ( f\circ\gamma)\in\mathcal{L}^p([a,b])$.
Therefore $f\circ\gamma\in\LLeb pabF$.

Since
	\begin{align*}
		\|f\circ\gamma\|_{\mathcal{L}^p,q}
		=\|\gamma\|_{\mathcal{L}^p,q\circ f},
	\end{align*}
the linear function $\LLeb pabf$ is continuous.
\end{proof}

\begin{remark}\label{rem:Lp-product}
From Lemma \ref{lem:cont-linear-Lp} we can easily conclude that for locally
convex spaces $E$ and $F$ we have
	\begin{align*}
		\LLeb pab{E\times F} \cong \LLeb pabE \times \LLeb pabF
	\end{align*}
as locally convex spaces. In fact, the function
	\begin{align*}
		\LLeb pab{E\times F} \to \LLeb pabE \times \LLeb pabF,\quad
		\gamma\mapsto (\pr_1\circ\gamma, \pr_2\circ\gamma)
	\end{align*}
is continuous linear (where $\pr_1, \pr_2$ are the projections on the first, resp., second component
of $E\times F$)
and is a linear bijection with the continuous inverse
	\begin{align*}
		\LLeb pabE \times \LLeb pabF \to \LLeb pab{E\times F},\quad
		(\gamma_1,\gamma_2)\mapsto \lambda_1\circ\gamma_1 + \lambda_2\circ\gamma_2,
	\end{align*}
where $\lambda_1\colon E\to E\times F, x\mapsto (x,0)$ and 
$\lambda_2\colon F\to E\times F, y\mapsto (0,y)$ are continuous and linear.
\end{remark}

As in \cite[Remark 3.7]{GlMR}, the following holds:

\begin{lemma}\label{lem:affine-linear-Lp}
Let $E$ be a locally convex space, let $a\leq\alpha<\beta\leq b$ and
	\begin{align*}
		f\colon [c,d]\to[a,b],\quad f(t):=\alpha + \frac{t-c}{d-c}(\beta-\alpha).
	\end{align*}
If $\gamma\in\LLeb pabE$ for $p\in[1,\infty]$, then $\gamma\circ f\in\LLeb pcdE$ and the function
	\begin{align*}
		\mathcal{L}^p(f,E)\colon\LLeb pabE\to\LLeb pcdE,\quad \gamma\mapsto \gamma\circ f
	\end{align*}
is continuous and linear.
\end{lemma}

\begin{proof}
Note that the composition $\gamma\circ f$ is measurable (by \cite[Theorem 3]{PonomZero},
the function $f$ has the property required in Lemma \ref{lem:lusin-special-cont-comp}).

Assume first $p<\infty$. By \cite[Satz 19.4]{B92}, the function $q^p\circ(\gamma\circ f)$
is $p$-integrable for each continuous seminorm $q$ on $E$, and
	\begin{align}\label{eq:affine-linear-p-integral}
		\int_c^d q(\gamma(f(t)))^p \,dt 
		= \frac{d-c}{\beta-\alpha}\int_{f(c)}^{f(d)} q(\gamma(t))^p \,dt < \infty,
	\end{align}
hence $\gamma\circ f\in\LLeb pcdE$. Furthermore, we see that
	\begin{align}\label{eq:affine-linear-p-seminorm}
		\| \gamma\circ f\|_{\mathcal{L}^p,q} \leq \left(\frac{d-c}{\beta-\alpha}\right)^{\frac{1}{p}} \|\gamma\|_{\mathcal{L}^p,q},
	\end{align}
whence the linear function $\mathcal{L}^p(f,E)$ is continuous.

Now, assume $p=\infty$. Then for every continuous seminorm $q$ on $E$
we have
	\begin{align*}
		\esssup_{t\in[c,d]} q(\gamma(f(t))) \leq \esssup_{t\in[a,b]} q(\gamma(t)) < \infty,
	\end{align*}
that is, $\gamma\circ f\in\LLeb {\infty}cdE$ and
	\begin{align}\label{eq:affine-linear-infty-seminorm}
		\| \gamma\circ f\|_{\mathcal{L}^{\infty},q} \leq \|\gamma\|_{\mathcal{L}^{\infty},q},
	\end{align}
hence the linear map $\mathcal{L}^{\infty}(f,E)$ is continuous and the proof is finished.
\end{proof}

As in \cite[3.15]{GlMR} (called the "locality axiom" there) the following holds:

\begin{lemma}\label{lem:Lp-locality-axiom}
For any $\abpart$, the function
	\begin{align*}
		\Gamma_E\colon \Leb pabE\to \prod_{j=1}^n \Leb p{t_{j-1}}{t_j}E,\quad
		[\gamma]\mapsto \left([\gamma|_{[t_{j-1},t_j]}]\right)_{j=1\ldots,n}
	\end{align*}
is an isomorphism of topological vector spaces.
\end{lemma}

\begin{proof}
The function $\Gamma_E$  is defined and
continuous, by Lemma \ref{lem:affine-linear-Lp}. For surjectivity, 
let $([\gamma_1],\ldots,[\gamma_n])\in\prod_{j=1}^n \Leb p{t_{j-1}}{t_j}E$
and define $\gamma\colon [a,b]\to E$ via
	\begin{align*}
		\gamma(t):=\gamma_j(t)\mbox{, if }t\in[t_{j-1},t_j[,\quad
		\gamma(t):=\gamma_n(t)\mbox{, if }t\in[t_{n-1},t_n].
	\end{align*}
The obtained map $\gamma$ is measurable (see Remark \ref{rem:lusin-meas-special-case})
and for Borel measurable $\bar{\gamma}$ as in Lemma \ref{lem:lusin-then-borel} and $p<\infty$ we have
	\begin{align}\label{eq:Lp-locality-integral}
		\int_a^b q(\bar{\gamma}(t))^p \,dt = \sum_{j=1}^n \int_{t_{j-1}}^{t_j} q(\gamma_j(t))^p \,dt<\infty.
	\end{align}
Further, for $p=\infty$ we have
	\begin{align}\label{eq:Lp-locality-ess-sup}
		\esssup_{t\in[a,b]}q(\bar{\gamma}(t)) = \max_{j=1,\ldots,n}\esssup_{t\in[t_{j-1},t_j]}q(\gamma_j(t))<\infty.
	\end{align}
Thus, in any case $[\bar{\gamma}]\in\Leb pabE$. As $\Gamma_E([\bar{\gamma}]) = ([\gamma_1],\ldots,[\gamma_n])$,
the function $\Gamma_E$ is surjective. As it is obviously injective and linear, $\Gamma_E$
is a continuous isomorphism of vector spaces and the continuity of the inverse $\Gamma_E^{-1}$ 
follows easily from Equations (\ref{eq:Lp-locality-integral}), resp., (\ref{eq:Lp-locality-ess-sup}).
\end{proof}

Furthermore, the $\mathcal{L}^p$-spaces have the \emph{subdivision property} (\cite[Lemma 5.26]{GlMR}).

\begin{lemma}\label{lem:Lp-subdivision}
Let $E$ be a locally convex space, let $\gamma\in\LLeb p01E$. For $n\in\N$ and $k\in\{0,\ldots,n-1\}$
define
	\begin{align}\label{eq:subdivision-function}
		\gamma_{n,k}\colon [0,1]\to E,\quad \gamma_{n,k}(t):=\frac{1}{n}\gamma\left(\frac{k+t}{n}\right).
	\end{align}
Then $\gamma_{n,k}\in\LLeb p01E$ for every $n,k$ and
	\begin{align*}
		\lim_{n\to\infty} \max_{k\in\{0,\ldots,n-1\}}\|\gamma_{n,k}\|_{\mathcal{L}^p,q} = 0
	\end{align*}
for each continuous seminorm $q$ on $E$.

More generally, the same holds for $\gamma\in\LLeb pabE$ and 
	\begin{align*}
		\gamma_{n,k}\colon[a,b]\to E,\quad \gamma_{n,k}(t):=\frac{1}{n}\gamma\left(a+\frac{k(b-a)+t-a}{n}\right).
	\end{align*}
\end{lemma}

\begin{proof}
The functions $f_{n,k}\colon [0,1]\to[\nicefrac{k}{n},\nicefrac{k+1}{n}], 
f_{n,k}(t):=\nicefrac{k+t}{n}$ are as in Lemma \ref{lem:affine-linear-Lp},
hence $\gamma_{n,k} = \nicefrac{1}{n}(\gamma\circ f_{n,k})\in\LLeb p01E$.

Further, for fixed $n\in\N$ and $p=\infty$ we have
	\begin{align*}
		\|\gamma_{n,k}\|_{\mathcal{L}^{\infty},q} 
		= \frac{1}{n}\|\gamma\circ f_{n,k}\|_{\mathcal{L}^{\infty},q}
		\leq \frac{1}{n}\|\gamma\|_{\mathcal{L}^{\infty},q}
	\end{align*}
for every continuous seminorm $q$ on $E$ and every $k\in\{0,\ldots,n-1\}$, 
by (\ref{eq:affine-linear-infty-seminorm}).
Hence
	\begin{align*}
		\max_{k\in\{0,\ldots,n-1\}}\|\gamma_{n,k}\|_{\mathcal{L}^{\infty},q}
		\leq \frac{1}{n}\|\gamma\|_{\mathcal{L}^{\infty},q} \to 0
	\end{align*}
as $n\to\infty$.

Now, if $2\leq p<\infty$, then for $n\in\N$ and a continuous seminorm $q$ on $E$ we have
	\begin{align*}
		\|\gamma_{n,k}\|_{\mathcal{L}^p,q}
		= \frac{1}{n}\|\gamma\circ f_{n,k}\|_{\mathcal{L}^p,q}
		\leq \frac{n^{\frac{1}{p}}}{n}\|\gamma\|_{\mathcal{L}^p,q}
		=n^{\frac{1}{p}-1}\|\gamma\|_{\mathcal{L}^p,q},
	\end{align*}
for each $k\in\{0,\ldots,n-1\}$, by (\ref{eq:affine-linear-p-seminorm}). Hence
	\begin{align*}
		\max_{k\in\{0,\ldots,n-1\}}\|\gamma_{n,k}\|_{\mathcal{L}^p,q}
		\leq n^{\frac{1}{p}-1}\|\gamma\|_{\mathcal{L}^p,q} \to 0
	\end{align*}
as $n\to\infty$.

Finally, let $p=1$. Fix some $\varepsilon>0$ and a continuous seminorm $q$ on $E$.
Each of the sets
	\begin{align*}
		A_m := \{t\in[a,b] : q(\gamma(t))>m
	\end{align*}
are in $\Borel {[0,1]}$ and
	\begin{align*}
		\lim_{m\to\infty} \int_{A_m} q(\gamma(t)) \, dt
		=\int_{\bigcap_{m\in\N}A_m} q(\gamma(t)) \, dt = 0,
	\end{align*}
because $(A_m)_{m\in\N}$ is a decreasing sequence and $\bigcap_{m\in\N}A_m = \emptyset$.
Therefore, for some $m\in\N$ we have
	\begin{align*}
		\int_{A_m} q(\gamma(t)) \, dt < \frac{\varepsilon}{2}.
	\end{align*}
We fix some $N\in\N$ such that $\nicefrac{m}{N} < \nicefrac{\varepsilon}{2}$ and for every $n\geq N$
we define
	\begin{align*}
		A_{n,k} := \{t\in[0,1] : f_{n,k}(t)\in A_m\}
	\end{align*}
Then
	\begin{align*}
		\int_{A_{n,k}} q(\gamma_{n,k}(t)) \,dt
		= \frac{1}{n} \int_{A_{n,k}} q(\gamma(f_{n,k}(t))) \,dt
		=\int_{f_{n,k}(A_{n,k})} q(\gamma(t)) \,dt,
	\end{align*}
by Equation (\ref{eq:affine-linear-p-integral}). Since
$f_{n,k}(A_{n,k}) = A_m\cap [\nicefrac{k}{n},\nicefrac{k+1}{n}]$, we obtain
	\begin{align*}
		\int_{f_{n,k}(A_{n,k})} q(\gamma(t)) \,dt
		\leq \int_{A_m} q(f(t)) \,dt
		< \frac{\varepsilon}{2},
	\end{align*}
by the choice of $m$.
Further
	\begin{align*}
		\|\gamma_{n,k}\|_{\mathcal{L}^1,q} 
		=\int_0^1 q(\gamma_{n,k}(t)) \,dt
		=\int_{A_{n,k}} q(\gamma_{n,k}(t)) \,dt + \int_{[0,1]\setminus A_{n,k}} q(\gamma_{n,k}(t)) \,dt
		<\varepsilon,
	\end{align*}
because $q(\gamma_{n,k}(t))=\nicefrac{1}{n}q(\gamma(f_{n,k}(t)))\leq \nicefrac{m}{n}<\nicefrac{\varepsilon}{2}$
for $t\in [0,1]\setminus A_{n,k}$.
Consequently,
	\begin{align*}
		\max_{k\in\{0,\ldots,n-1\}} \|\gamma_{n,k}\|_{\mathcal{L}^1,q} <\varepsilon,
	\end{align*}
in other words, $\max_{k\in\{0,\ldots,n-1\}}\|\gamma_{n,k}\|_{\mathcal{L}^1,q} \to 0$ as $n\to\infty$,
as required.
\end{proof}

The following result can be found in \cite[Lemma 2.1]{GlMR} (for suitable vector spaces).

\begin{lemma}\label{lem:linear2-in-Lp}
Let $X$ be a topological space, $U\subseteq X$ be an open subset and $E$, $F$ be locally convex spaces.
Let $f\colon U\times E\to F$ be continuous and linear in the second argument.
If $\eta\in C([a,b],U)$ and $\gamma\in\LLeb pabE$ for $p\in[1,\infty]$, 
then $f\circ (\eta,\gamma)\in\LLeb pabF$.
\end{lemma}

\begin{proof}
By Lemma \ref{lem:lusin-into-products} and Remark \ref{rem:lusin-cont-comp}, the composition
$f\circ(\eta,\gamma)$ is a measurable function.

Now, consider the continuous function
	\begin{align*}
		h_\eta\colon [a,b]\times E\to F,\quad
		h_\eta(t,v):=f(\eta(t),v).
	\end{align*}
Let $q$ be a continuous seminorm on $F$. Then
$h_\eta([a,b]\times\{0\})=\{0\}\subseteq B_1^q(0)$,
thus $[a,b]\times\{0\}\subseteq V$, where $V:=h_\eta^{-1}(B_1^q(0))$ is an open subset of $[a,b]\times E$. 
Using the Wallace Lemma,
we find an open subset $W\subseteq E$ such that 
$[a,b]\times\{0\}\subseteq [a,b]\times W\subseteq V$.
Then there is a continuous seminorm $\pi$ on $E$ such that
	\begin{align*}
		[a,b]\times\{0\} \subseteq [a,b]\times \overline{B_1^\pi(0)} \subseteq [a,b]\times W
		\subseteq V.
	\end{align*}
	
We show that for each $(t,v)\in[a,b]\times E$ we have $q(h_\eta(t,v))\leq \pi(v)$.
In fact, if $\pi(v)>0$, then (using the linearity of $f$ in $v$) we have 
$(\nicefrac{1}{\pi(v)})q(h_\eta(t,v)) = q(h_\eta(t,(\nicefrac{1}{\pi(v)})v))\leq 1$.
If $\pi(v)=0$, then for each $r>0$ we have $rv\in\overline{B_1^\pi(0)}$, whence
$rq(h_\eta(t,v))=q(h_\eta(t,rv))\leq 1$, that is $q(h_\eta(t,v))\leq\nicefrac{1}{r}$, consequently 
$q(h_\eta(t,v))=0=\pi(v)$.

Now, if $p<\infty$, then
	\begin{align*}
		\int_a^b q(f(\eta(t),\gamma(t)))^p \,dt
		=\int_a^b q(h_\eta(t,\gamma(t)))^p \,dt
		\leq \int_a^b \pi(\gamma(t))^p \,dt
		<\infty,
	\end{align*}
thus $q\circ (f\circ(\eta,\gamma))\in\mathcal{L}^p([a,b])$.

If $p=\infty$, then $q(f(\eta(t),\gamma(t)))\leq\pi(\gamma(t))$, whence 
	\begin{align*}
		\esssup_{t\in[a,b]}(q(f(\eta(t),\gamma(t))))\leq \esssup_{t\in[a,b]}(\pi(\gamma(t)))<\infty,
	\end{align*}
thus
$q\circ (f\circ(\eta,\gamma))\in\mathcal{L}^\infty([a,b])$.
\end{proof}

The following lemma (\cite[Lemma 2.4]{GlMR}) will be used in the proof of Proposition
\ref{prop:double-push-forward-Lp}.

\begin{lemma}\label{lem:triple-push-forward-Lp}
Let $E_1$, $E_2$, $E_3$ and $F$ be locally convex spaces, $U\subseteq E_1$, $V\subseteq E_2$
be open subsets and the function $f\colon U\times V\times E_3\to F$ be a $C^1$-function
and linear in the third argument. Then the function
	\begin{align*}
		\tilde{f}\colon U\times C([a,b],V)\times \Leb pab{E_3} &\to \Leb pabF,\\
		(u,\eta,[\gamma])&\mapsto [f(u,\bullet)\circ (\eta,\gamma)]
	\end{align*}
is continuous.(Here $C([a,b],V)$ is endowed with the topology of uniform convergence.)
\end{lemma}

\begin{proof}
Fix some $(\bar{u}, \bar{\eta},[\bar{\gamma}])\in  U\times C([a,b],V)\times \Leb pab{E_3}$
and let $q$ be a continuous seminorm on $F$. The subset 
$K:=\{\bar{u}\}\times\bar{\eta}([a,b])\subseteq U\times V$ is compact, hence from Lemma
\cite[Lemma 1.61]{GlMR} follows that there are seminorms $\pi$ on $E_1\times E_2$
and $\pi_3$ on $E_3$ such that $K + B_1^{\pi}(0)\subseteq U\times V$ and
	\begin{align*}
		q(f(u,v,w)-f(u',v',w'))\leq \pi_3(w-w') + \pi(u-u', v-v')\pi_3(w')
	\end{align*}
for all $(u,v),(u',v')\in K+B_1^{\pi}(0)$, $w,w'\in E_3$.
We may assume that $\pi(x,y) = \max\{\pi_1(x), \pi_2(y)\}$ for some continuous
seminorms $\pi_1$ on $E_1$, $\pi_2$ on $E_2$. Then, setting
	\begin{align*}
		U_0:=B_1^{\pi_1}(\bar{u}),\quad V_0:=\bar{\eta}([a,b])+B_1^{\pi_2}(0),
	\end{align*}
we define an open neighborhood
	\begin{align*}
		\Omega := U_0\times C([a,b],V_0)\times\Leb pab{E_3}
	\end{align*}	 
of $(\bar{u}, \bar{\eta},[\bar{\gamma}])$ and show that if 
$(u,\eta,[\gamma])\to (\bar{u}, \bar{\eta},[\bar{\gamma}])$ in $\Omega$, then
$\tilde{f}(u,\eta,[\gamma])\to\tilde{f}(\bar{u}, \bar{\eta},[\bar{\gamma}])$ in $\Leb pab{E_3}$,
because
	\begin{align*}
		\|\tilde{f}(u,\eta,[\gamma]) &- \tilde{f}(\bar{u}, \bar{\eta},[\bar{\gamma}])\|_{L^p,q}\\
		&\leq \|[\gamma-\bar{\gamma}]\|_{L^p,\pi_3}
		+\max\{\pi_1(u-\bar{u}), \|\eta-\bar{\eta}\|_{\infty,\pi_2}\}\|[\bar{\gamma}]\|_{L^p,\pi_3}\to 0.
	\end{align*}
In other words, $\tilde{f}$ is continuous in $(\bar{u}, \bar{\eta},[\bar{\gamma}])$.
\end{proof}

Before investigating differentiable mappings between Lebesgue spaces,
we recall some details concerning the concept of the differentiability on locally convex spaces.
The concept we work with goes back to Bastiani \cite{Bastiani} and is well known
as  Keller's $C_c^k$-calculus \cite{Keller}. The results below can be found, for example, in
\cite{GlHRestr}, \cite{BGlNDiff} and many others.

\begin{definition}
Let $E$, $F$ be locally convex spaces, let $f\colon U\to F$ be a function 
defined on an open subset $U\subseteq E$. The map $f$ is called a \emph{$C^0$-map} if it is
continuous; it is called a \emph{$C^1$-map} if it is $C^0$ and for every $x\in U$, $y\in E$ the directional derivative
	\begin{align*}
		df(x,y):= \lim_{h\to 0}\frac{f(x+hy(-f(x)}{h}
	\end{align*}
exists in $F$ and the obtained differential
	\begin{align*}
		df\colon U\times E\to F,\quad (x,y)\mapsto df(x,y)
	\end{align*}
is continuous. Further, for $k\geq 2$, a continuous map $f$ is called $C^k$ if
it is $C^1$ and $df$ is $C^{k-1}$. 
Finally, a continuous function $f$ is called a \emph{$C^{\infty}$-map} or \emph{smooth},
if $f$ is $C^k$ for every $k\in\N$.

If $f\colon [a,b]\to F$ is a function defined on an interval $[a,b]\subseteq\R$, then $f$ is called $C^0$ if
it is continuous, and it is called $C^1$ if for every $x\in[a,b]$ the (possibly one-sided) limit
	\begin{align*}
		f^\prime(x):=\lim_{h\to 0}\frac{f(x+h)-f(x)}{h}
	\end{align*}
exists and the obtained derivative
	\begin{align*}
		f^\prime\colon[a,b]\to F
	\end{align*}
is continuous.
\end{definition}

Several properties of differentiable functions will be used repeatedly without further mention.

\begin{remark}
\begin{itemize}
\item[(i)]
Each continuous, linear function is smooth.
\item[(ii)]
The differential $df\colon U\times E\to F$ of a $C^1$-map is linear in the second argument.
\item[(iii)]
If $f\colon U\to F$, $g\colon F\to G$ are $C^k$-maps between locally 
convex spaces, then the composition $g\circ f\colon U\to G$ is $C^k$. Furthermore,
if $g$ is a linear topological embedding
such that $g(F)$ is closed in $G$ and if $g\circ f\colon U\to G$ is $C^k$, then $f$ is $C^k$.
\item[(iv)]
A function $f\colon U\to\prod_{j\in J}F_j$ (where $F_j$ are locally convex spaces) 
is $C^k$ if and only if each of the components
$\pr_j\circ f\colon U\to F_j$ is $C^k$.
\end{itemize}
\end{remark}

\begin{remark}\label{rem:f[1]-map}
For an open subset $U\subseteq E$ and a $C^1$-map $f\colon U\to F$ define the open subset
	\begin{align*}
		U^{[1]}:=\{(x,y,h)\in U\times E\times \R : x+hy\in U\}
	\end{align*}
and the function $f^{[1]}\colon U^{[1]}\to F$ via
	\begin{align*}
		f^{[1]}(x,y,h):=\frac{f(x+hy)-f(x)}{h},\quad\mbox{ for }h\neq 0,
	\end{align*}
and
	\begin{align*}
		f^{[1]}(x,y,0) = df(x,y).
	\end{align*}
Then $f^{[1]}$ is continuous. 

Conversely, if for $f\colon U\to F$ there is a continuous map $f^{[1]}\colon U^{[1]}\to F$
such that $f^{[1]}(x,y,h)=\frac{f(x+hy)-f(x)}{h}$ for $h\neq 0$, then $f$ is a $C^1$-function and
$df(x,y)=f^{[1]}(x,y,0)$.
\end{remark}

Returning to our theory, we show that as in \cite[Proposition 2.3]{GlMR}, the following holds:

\begin{proposition}\label{prop:double-push-forward-Lp}
Let $E_1$, $E_2$, $F$ be locally convex spaces, let $V\subseteq E_1$ be open and the
function $f\colon V\times E_2\to F$ be $C^{k+1}$ for $k\in\N\cup\{0,\infty\}$ 
and linear in the second argument. Then for $p\in[1,\infty]$
the function
	\begin{align*}
		\Theta_f\colon C([a,b],V)\times\Leb pab{E_2}&\to\Leb pabF, \\
		(\eta,[\gamma])&\mapsto [f\circ(\eta,\gamma)]
	\end{align*}
is $C^k$.
\end{proposition}

\begin{proof}
For $k=0$, the assertion holds by Lemma \ref{lem:triple-push-forward-Lp}.
Further, we may assume $k<\infty$ and proceed by induction.

\emph{Induction start: $k=1$.} The map $\Theta_f$ is continuous by the previous step; we show that
for all $(\eta,[\gamma])\in C([a,b],V)\times\Leb pab{E_2}$ and
$(\bar{\eta},[\bar{\gamma}])\in C([a,b],E_1)\times\Leb pab{E_2}$ the directional
derivative 
	\begin{align*}
		d(\Theta_f)(\eta,[\gamma],\bar{\eta},[\bar{\gamma}]) :=
		\lim_{h\to 0}\frac{\Theta_f(\eta+h\bar{\eta},[\gamma+h\bar{\gamma}])-\Theta_f(\eta,[\gamma])}{h}
	\end{align*} 
exists in $\Leb pabF$ and equals $[df\circ (\eta,\gamma,\bar{\eta},\bar{\gamma})]$.

Given $\eta,[\gamma],\bar{\eta},[\bar{\gamma}]$ as above, we note that $\eta([a,b])$ is a compact subset
of the open subset $V\subseteq E_1$, thus there exists an open $0$-neighborhood $U\subseteq E_1$
such that $\eta([a,b])+U\subseteq V$. Further, there is some balanced $0$-neighborhood $W\subseteq U$ such
that $W+W\subseteq U$. As $\bar{\eta}([a,b])$ is bounded in $E_1$ (being compact), for some $\varepsilon>0$
we have $\bar{\eta}([a,b])\subseteq \nicefrac{1}{\varepsilon}W$. In this manner we obtain an open subset
	\begin{align*}
		\Omega:=]-\varepsilon,\varepsilon[ \times (\eta([a,b])+W) \times \frac{1}{\varepsilon}W \times E_2\times E_2
		\subseteq \R\times V\times E_1\times E_2\times E_2
	\end{align*}
for which holds $]-\varepsilon,\varepsilon[\times\eta([a,b])\times\bar{\eta}([a,b])\times\gamma([a,b])\times\bar{\gamma}([a,b])
\subseteq \Omega$ and for all $(t,w,\bar{w},x,\bar{x})\in \Omega$ we have $(w+t\bar{w},x+t\bar{x})\in V\times E_2$
(that is, $\Omega$ is contained in $(V\times E_2)^{[1]}$ constructed as in Remark \ref{rem:f[1]-map}).

Now, for $f^{[1]}\colon (V\times E_2)^{[1]}\to F$ the function
	\begin{align*}
		\Omega\to F,\quad (t,w,\bar{w},x,\bar{x})\mapsto f^{[1]}(w,x,\bar{w},\bar{x},t)
	\end{align*}	 
is $C^1$ and linear in $(x,\bar{x})$, thus from Lemma \ref{lem:triple-push-forward-Lp} follows
(identifying the space $\Leb pab{E_2\times E_2}$ with $\Leb pab{E_2}\times\Leb pab{E_2}$, see
Remark \ref{rem:Lp-product}) that
	\begin{align*}
		(t,\varphi,\bar{\varphi},[\psi],[\bar{\psi}])\mapsto 
		[f^{[1]}(\bullet,t)\circ(\varphi,\psi,\bar{\varphi},\bar{\psi})]\in \Leb pabF
	\end{align*}
is continuous on
	\begin{align*}
		]-\varepsilon,\varepsilon[
		\times C([a,b],\eta([a,b])+W)
		\times C([a,b],\nicefrac{1}{\varepsilon}W)
		\times\Leb pab{E_2}
		\times\Leb pab{E_2}.
	\end{align*}
Hence
	\begin{align*}
		]-\varepsilon,\varepsilon[\to\Leb pabF,\quad 
		t\mapsto [f^{[1]}(\bullet,t)\circ(\eta,\gamma,\bar{\eta},\bar{\gamma})]
	\end{align*}
is continuous. It follows
	\begin{align*}
		d(\Theta_f)(\eta,[\gamma],\bar{\eta},[\bar{\gamma}])
		&=\lim_{h\to0}\frac{1}{h}\left(\Theta_f(\eta+t\bar{\eta},[\gamma+t\bar{\gamma}])-\Theta_f(\eta,[\gamma])\right)\\
		&=\lim_{h\to0}\frac{1}{t}\left([f\circ(\eta+t\bar{\eta},\gamma+t\bar{\gamma})]-[f\circ(\eta,\gamma)]\right)\\
		&=\lim_{h\to0}[f^{[1]}(\bullet,h)\circ(\eta,\gamma,\bar{\eta},\bar{\gamma})]\\
		&=[f^{[1]}(\bullet,0)\circ(\eta,\gamma,\bar{\eta},\bar{\gamma})]
		=[df\circ(\eta,\gamma,\bar{\eta},\bar{\gamma})]
	\end{align*}
in $\Leb pabF$.

It remains to show that 
	\begin{align*}
		d(\Theta_f)\colon C([a,b],V)\times \Leb pab{E_2}\times C([a,b],E_1)\times \Leb pab{E_2}\to\Leb pabF
	\end{align*}
is continuous.
But as the function
	\begin{align}\label{eq:help-df}
		V\times E_1\times E_2\times E_2\to F,\quad (w,\bar{w},x,\bar{x})\mapsto df(w,x,\bar{w},\bar{x})
	\end{align}
is $C^1$ and linear in $(x,\bar{x})$, by induction start
	\begin{align*}
		C([a,b],V)\times C([a,b],E_1)\times\Leb pab{E_2}\times\Leb pab{E_2}&\to\Leb pabF,\\
		\quad (\varphi,\bar{\varphi},[\psi],[\bar{\psi}])&\mapsto [df\circ(\varphi,\psi,\bar{\varphi},\bar{\psi})]
	\end{align*}
is continuous (we identify the $L^p$-spaces again, as above), 
hence $d(\Theta_f)$ is continuous. Therefore, $\Theta_f$ is $C^1$.

\emph{Induction step:} Now, assume that $f$ is $C^{k+2}$. Then $\Theta_f$ is $C^1$ by induction start and 
$df$ is $C^{k+1}$. Then the map in (\ref{eq:help-df}) is $C^{k+1}$ and linear in $(x,\bar{x})$, hence
by induction hypothesis, the map 
$(\varphi,\bar{\varphi},[\psi],[\bar{\psi}])\mapsto [df\circ(\varphi,\psi,\bar{\varphi},\bar{\psi})]
=d(\Theta_f)(\varphi,[\psi],\bar{\varphi},[\bar{\psi}])$ is $C^k$. Hence $\Theta_f$ is $C^{k+1}$.
\end{proof}

\subsection{Integrable $\mathbf{\mathcal{L}^p}$-functions}
\label{subsec:weak-integrable-functions}

\begin{definition}\label{def:weak-integral}
Let $E$ be a locally convex space and let $\gamma\colon[a,b]\to E$ be such that 
$\alpha\circ\gamma\in\mathcal{L}^1([a,b])$ for every continuous linear form $\alpha\in E'$. 
If there exists some $w\in E$ such that 
	\begin{align*}
		\alpha(w)=\int_a^b \alpha(\gamma(t)) \,dt
	\end{align*}
for every $\alpha$, then $w$ is called the \emph{weak integral of $\gamma$ from $a$ to $b$}, and
we write $\int_a^b \gamma(t) \,dt:=w$. As the continuous linear forms separate the points on $E$,
the weak integral of a function $\gamma$ is unique if it exists.
\end{definition}

\begin{remark}\label{rem:conditions-weak-integral}
Since $|\alpha|$ is a continuous seminorm on $E$ for $\alpha\in E'$, each $\gamma\in\LLeb 1abE$
satisfies the condition $\alpha\circ\gamma\in\mathcal{L}^1([a,b])$. Further, if $\int_a^b \gamma(t) \,dt$
exists in $E$, then for every continuous seminorm $q$ on $E$ we have
	\begin{align*}
		q\left(\int_a^b \gamma(t) \,dt\right) \leq \int_a^b q(\gamma(t)) \,dt.
	\end{align*}
\end{remark}

In \cite[Lemma 1.19, Lemma 1.23 and Lemma 1.43]{GlMR}, the author proves that each
$\gamma\in\mathcal{L}_B^1([a,b],E)$, resp, $\gamma\in\mathcal{L}_{rc}^{\infty}([a,b],E)$
has a weak integral in $E$ for suitable spaces $E$. We show below that sequential completeness
of $E$ suffices for each $\mathcal{L}^1$-function to be weak integrable.

\begin{proposition}\label{prop:l1-function-weak-integral}
Let $E$ be a sequentially complete locally convex space. Then each $\gamma\in\LLeb 1abE$
has a weak integral $\int_a^b \gamma(t) \, dt\in E$.

Further, the function
	\begin{align}\label{eq:weak-integral-function}
		\eta\colon[a,b]\to E,\quad \eta(t):= \int_a^t \gamma(s) \, ds
	\end{align}
is continuous.
\end{proposition}

\begin{proof}
As $\gamma$ is measurable, pick a disjoint sequence $(K_n)_{n\in\N}$ of compact sets
$K_n\subseteq[a,b]$ such that $\gamma|_{K_n}$ is continuous and
$\lambda([a,b]\setminus\bigcup_{n\in\N}K_n)=0$. Then for each $n\in\N$
the weak integral $\int_{K_n}\gamma(t) \, dt$ exists in $E$ (by \cite[3.27 Theorem]{RudinFA}). 
Then
	\begin{align*}
		\sum_{n=1}^\infty q\left(\int_{K_n} \gamma(t) \,dt\right)
		\leq \sum_{n=1}^\infty \int_{K_n} q(\gamma(t)) \,dt = \int_a^b q(\gamma(t)) \,dt < \infty,
	\end{align*}
that is, the series $\sum_{n=1}^\infty \int_{K_n} \gamma(t) \, dt$ is absolutely convergent in $E$,
hence convergent, as $E$ is assumed sequentially complete. 

Next, we see that
$\sum_{n=1}^\infty \int_{K_n} \gamma(t) \, dt = \int_a^b \gamma(t) \, dt$,
because for every continuous linear form $\alpha\in E'$ we have
	\begin{align*}
		\alpha\left(\sum_{n=1}^\infty \int_{K_n} \gamma(t) \, dt\right)
		&=\sum_{n=1}^\infty \alpha\left(\int_{K_n} \gamma(t) \, dt\right)\\
		&=\sum_{n=1}^\infty \int_{K_n} \alpha(\gamma(t)) \,dt
		=\int_a^b  \alpha(\gamma(t)) \,dt.
	\end{align*}
	
To prove the continuity of $\eta$ in every $t\in [a,b]$, let $q$ be a continuous seminorm on $E$ 
and let $\varepsilon>0$. Then there exists some $\delta>0$ such that whenever $|t-r|<\delta$,
we have $\int_r^t q(\gamma(s)) \,ds < \varepsilon$ (follows from the classical Fundamental
Theorem of Calculus, see \cite[VII. 4.14]{Elst}). Therefore
	\begin{align*}
		q(\eta(t)-\eta(r)) 
		&= q\left(\int_a^t \gamma(s) \,ds - \int_a^r \gamma(s) \,ds\right)\\
		&=q\left(\int_r^t \gamma(s) \,ds\right)
		\leq \int_r^t q(\gamma(s)) \,ds < \varepsilon,
	\end{align*}
whence $\eta$ is continuous in $t$.
\end{proof}

The differentiability (almost everywhere) of $\eta$ defined in (\ref{eq:weak-integral-function})
is shown in \cite[Lemma 1.28]{GlMR} in the \Frechet case.
In the next proposition, we get (as in \cite[\S 5]{Bochner}) a similar result for
$E$ merely metrizable.

\begin{proposition}\label{prop:weak-integral-diff}
Let $E$ be a metrizable locally convex space, let $\gamma\in\LLeb 1abE$. If the function
	\begin{align*}
		\eta\colon [a,b]\to E,\quad \eta(t):=\int_a^t \gamma(s) \,ds
	\end{align*}
is everywhere defined, then $\eta$ is continuous and for almost every $t\in[a,b]$ the derivative $\eta^\prime(t)$ exists
and equals $\gamma(t)$. 
\end{proposition}

\begin{proof}
Note that the continuity of $\eta$ can be shown as in the proof of Proposition \ref{prop:l1-function-weak-integral}.
We may assume that $\gamma(t)=0$ for each $t\notin \bigcup_{n\in\N} K_n$, where $(K_n)_{n\in\N}$ is a
sequence as in Definition \ref{def:lusin-meas}. Our aim is to show that for almost every $t\in[a,b]$ the difference quotient
	\begin{align*}
		\frac{1}{r}(\eta(t+r)-\eta(t)) 
		= \frac{1}{r}\left(\int_a^{t+r} \gamma(s) \,ds - \int_a^t \gamma(s) \, ds\right)
		= \frac{1}{r}\int_t^{t+r} \gamma(s) \,ds
	\end{align*}
tends to $\gamma(t)$ as $r\to 0$. 
That is, for every $\varepsilon>0$ and continuous seminorm $q_m$ on $E$ we have
	\begin{align}\label{eq:integral-diff-result}
		q_m\left(\frac{1}{r} \int_t^{t+r}\gamma(s)\,ds -\gamma(t)\right) 
		= q_m\left(\frac{1}{r} \int_t^{t+r}\gamma(s) - \gamma(t) \,ds\right) 
		< \varepsilon
	\end{align}
for $r\neq 0$ small enough.

We fix some $\varepsilon>0$ and some continuous seminorm $q_m$.
The set $\gamma([a,b])=\{0\}\cup\bigcup_{n\in\N}\gamma(K_n)\subseteq E$ is separable, say 
$\gamma([a,b]) = \overline{\{a_k : k\in\N\}}$.
Thus for every $t\in[a,b]$ we find some $a_m(t)$ such that
	\begin{align*}
		q_m(\gamma(t)-a_m(t))<\frac{1}{3}\varepsilon,
	\end{align*}
hence for every $r\neq 0$ small enough we have
	\begin{align}\label{eq:integral-diff-first}
		\frac{1}{r}\int_t^{t+r} q_m(\gamma(t)-a_m(t)) \,ds < \frac{1}{3}\varepsilon.
	\end{align}
Furthermore, each of the functions
	\begin{align*}
		h_{k,m}\colon[a,b]\to\R,\quad h_{k,m}(t):=q_m(\gamma(t)-a_k)
	\end{align*}
is in $\mathcal{L}^1([a,b])$, hence by the classical Fundamental Theorem of Calculus (see \cite[VII. 4.14]{Elst})
there exist some sets $N_{k,m}\subseteq [a,b]$ such that $\lambda(N_{k,m})=0$ and for every $t\notin N_{k,m}$
we have
	\begin{align}\label{eq:integral-diff-second}
		\left| \frac{1}{r}\int_t^{t+r} q_m(\gamma(s)-a_k) \,ds - q_m(\gamma(t)-a_k) \right| <\frac{1}{3}\varepsilon.
	\end{align}
for $r\neq0$ small enough.

Consequently, for $t\notin \bigcup_{m\in\N}N_{k,m}$ we have
	\begin{align*}
		\left| \frac{1}{r} \int_t^{t+r} q_m(\gamma(s)-\gamma(t)) \,ds \right| 
		&\leq \left| \frac{1}{r} \int_t^{t+r} q_m(\gamma(s)-\gamma(t)) - q_m(\gamma(s)-a_m(t)) \,ds\right|\\
			& + \left| \frac{1}{r} \int_t^{t+r} q_m(\gamma(s)-a_m(t)) - q_m(\gamma(t)-a_m(t)) \,ds\right|\\
			& + \left| \frac{1}{r} \int_t^{t+r} q_m(\gamma(t)-a_m(t)) \,ds \right|\\
			& < \left| \frac{1}{r} \int_t^{t+r} q_m(\gamma(s)-\gamma(t)) - q_m(\gamma(s)-a_m(t)) \,ds\right|\\
			& +\frac{1}{3}\varepsilon+\frac{1}{3}\varepsilon,
	\end{align*}
using the estimates in (\ref{eq:integral-diff-second}) and (\ref{eq:integral-diff-first}).
Finally,
	\begin{align*}
		&\left| \frac{1}{r} \int_t^{t+r} q_m(\gamma(s)-\gamma(t)) - q_m(\gamma(s)-a_m(t)) \,ds\right|\\
		&\leq \frac{1}{|r|} \int_t^{t+r} |q_m(\gamma(s)-\gamma(t)) - q_m(\gamma(s)-a_m(t))| \,ds\\
		&\leq \frac{1}{|r|} \int_t^{t+r} q_m(\gamma(s)-\gamma(t) - \gamma(s)+a_m(t)) \,ds\\
		&= \frac{1}{|r|} \int_t^{t+r} q_m(\gamma(t) - a_m(t)) \,ds
		<\frac{1}{3}\varepsilon,
	\end{align*}
using (\ref{eq:integral-diff-first}) again.

Altogether, we have
	\begin{align*}
		q_m\left(\frac{1}{r} \int_t^{t+r}\gamma(s) - \gamma(t) \,ds\right) 
		&= \frac{1}{|r|}q_m\left(\int_t^{t+r}\gamma(s) - \gamma(t) \,ds\right)\\
		&\leq \frac{1}{|r|}\int_t^{t+r}q_m(\gamma(s) - \gamma(t)) \,ds\\
		&\leq \left| \frac{1}{r} \int_t^{t+r} q_m(\gamma(s)-\gamma(t)) \,ds \right|
		<\varepsilon
	\end{align*}
by the above. Thus the desired estimate (\ref{eq:integral-diff-result}) holds for each $t\notin \bigcup_{m\in\N}N_{k,m}$
and $\lambda(\bigcup_{k,m\in\N}N_{k,m})=0$, whence the proof is finished.
\end{proof}

Even if $E$ is not metrizable (so that $\eta^\prime$ does not necessarily exist, not even almost everywhere), we show
that (as in \cite{GlMR}) the equivalence class $[\gamma]$ is still uniquely determined by $\eta$.

\begin{lemma}\label{lem:weak-integral-unique}
Let $E$ be a sequentially complete locally convex space and let $\gamma\in\LLeb 1abE$. If 
$\int_a^t \gamma(s) \,ds=0$
for all $t\in [a,b]$, then $\gamma(s)=0$ for almost all $s\in[a,b]$.
\end{lemma}

\begin{proof}
Let $\alpha$ be a continuous linear functional on $E$. Then we have
	\begin{align*}
		\int_a^t (\alpha\circ\gamma)(s) \,ds = \alpha\left(\int_a^t \gamma(s) \,ds\right) = 0
	\end{align*}	 
for every $t\in[a,b]$. From the Fundamental Theorem of Calculus (see \cite[VII. 4.14]{Elst})
follows that $(\alpha\circ\gamma)(t)=0$ for almost every $t\in[a,b]$. As $\alpha\in E'$ was arbitrary,
from Lemma \ref{lem:lusin-almost-zero} follows that $\gamma(t)=0$ for almost every $t$.
\end{proof}

\section{\emph{AC}-functions} 

\subsection{Vector-valued \emph{AC}-functions}

The vector spaces $\ACp abE$ are defined similarly to \cite[Definition 3.6]{GlMR}.

\begin{definition}\label{def:abs-cont-vector-space}
Let $E$ be a sequentially complete locally convex space. For $p\in[1,\infty]$ we denote by
$\ACp abE$ the vector space of continuous functions $\eta\colon[a,b]\to E$ such that
for some $\gamma\in\LLeb pabE$ we have 
	\begin{align*}
		\eta(t)=\eta(a)+\int_a^t \gamma(s) \, ds\quad\mbox{ for all }t\in[a,b].
	\end{align*}
As $\eta$ uniquely determines $[\gamma]$ (see Lemma \ref{lem:weak-integral-unique}),
we write $\eta^\prime:=[\gamma]$.

The function
	\begin{align}
		\Phi\colon \ACp abE \to E\times\Leb pabE,\quad
		\eta\mapsto (\eta(a),\eta^\prime)
	\end{align}
is an isomorphism of vector spaces and we endow $\ACp abE$ with the locally convex
topology making $\Phi$ a homeomorphism.

In other words, the topology on $\ACp abE$ is defined by seminorms
	\begin{align}\label{eq:AC-seminorms}
		\eta\mapsto q(\eta(a)),\quad \eta\mapsto \|\eta^\prime\|_{L^p,q}
	\end{align}
for continuous seminorms $q$ on $E$.
\end{definition}

\begin{remark}\label{rem:AC-vecsp-inclusions}
As in \cite[Remark 3.7]{GlMR}, we mention that
from Remark \ref{rem:Lp-inclusions} readily follows that for $1\leq p\leq r\leq\infty$ we have
	\begin{align*}
		AC_{L^{\infty}}([a,b],E)
		\subseteq AC_{L^r}([a,b],E)
		\subseteq AC_{L^p}([a,b],E)
		\subseteq AC_{L^1}([a,b],E),
	\end{align*}
with continuous inclusion maps.
\end{remark}

Also the following inclusion map is continuous (\cite[Lemma 3.9]{GlMR}).

\begin{lemma}\label{lem:AC-vecsp-top-embedding}
Let $E$ is a sequentially complete locally convex space and endow the vector space $C([a,b],E)$
with the topology of uniform convergence. 
Then for $p\in[1,\infty]$ the inclusion map
	\begin{align*}
		\incl\colon \ACp abE \to C([a,b],E),\quad \eta\mapsto \eta
	\end{align*}
is continuous.
\end{lemma}

\begin{proof}
Let $\eta\in\ACp abE$ and denote $\eta^\prime = [\gamma]$. For a continuous seminorm $q$ on $E$
and $t\in[a,b]$ we have
	\begin{align*}
		q(\eta(t))
		&=q\left(\eta(t_0)+\int_{t_0}^t \gamma(s) \,ds\right)
		\leq q(\eta(t_0)) + q\left(\int_{t_0}^t \gamma(s) \,ds\right)\\
		&\leq q(\eta(t_0)) + \int_{t_0}^t q(\gamma(s)) \,ds
		=q(\eta(t_0)) + \|\gamma\|_{\mathcal{L}^1,q}\\
		&\leq q(\eta(t_0)) + (b-a)^{1-\frac{1}{p}}\|\gamma\|_{\mathcal{L}^p,q}
		=q(\eta(t_0)) + (b-a)^{1-\frac{1}{p}}\|\eta^\prime\|_{L^p,q},
	\end{align*}
see Remark \ref{rem:Lp-inclusions}. Thus
	\begin{align*}
		\|\eta\|_{\infty,q} \leq q(\eta(t_0)) + (b-a)^{1-\frac{1}{p}}\|\eta^\prime\|_{L^p,q},
	\end{align*}
whence the (linear) inclusion map is continuous (recall that the topology on $C([a,b],E)$ is defined
by the family of seminorms $\|\eta\|_{\infty,q}:=\sup_{t\in[a,b]}q(\eta(t))$ with continuous seminorms
$q$ on $E$).
\end{proof}

\begin{remark}\label{rem:open-in-AC-vecsp}
For a subset $V\subseteq E$, denote by $\ACp abV$ the set of functions $\eta\in\ACp abE$
such that $\eta([a,b])\subseteq V$. Then, as a consequence of Lemma \ref{lem:AC-vecsp-top-embedding},
$\ACp abV = \incl^{-1}(C([a,b],V))$ is open in
$\ACp abE$ if $V$ is open in $E$.
\end{remark}

\begin{remark}\label{rem:eval-smooth-AC-vecsp}
It is well known that the evaluation map 
$C([a,b],E)\to E, \eta\mapsto \eta(\alpha)$ is continuous linear for $\alpha\in[a,b]$. By 
Lemma
\ref{lem:AC-vecsp-top-embedding}, so is the inclusion map $\incl\colon\ACp abE\to C([a,b],E)$, hence
the evaluation map
	\begin{align*}
		\ev_{\alpha}\colon \ACp abE\to E,\quad
	 	\eta\mapsto \eta(\alpha)
	\end{align*}
is continuous, linear.
\end{remark}

The next lemma (which is not difficult to prove) will be used in the proof of Lemma
\ref{lem:C1-maps-act-on-AC-vecsp} (which corresponds to \cite[Lemma 3.18]{GlMR}).

\begin{lemma}\label{lem:compact-continuity}
Let $E$ be a locally convex space, $U\subseteq E$ be an open subset and
$f\colon U\to \R$ be continuous. Then for every compact subset $K\subseteq U$
and every $\varepsilon>0$ there exists a continuous seminorm $q$ on $E$ such that
$K+B_1^q(0)\subseteq U$ and
	\begin{align*}
		|f(x)-f(y)| < \varepsilon\quad\mbox{ for }x\in K, y\in B_1^q(x).
	\end{align*}
\end{lemma}

\begin{lemma}\label{lem:C1-maps-act-on-AC-vecsp}
Let $E$, $F$ be sequentially complete locally convex spaces, $V\subseteq E$ be an open subset
and $p\in[1,\infty]$.
If $f\colon V\to F$ is a $C^1$-function then
	\begin{align*}
		f\circ\eta\in\ACp abF
	\end{align*}
for every $\eta\in\ACp abV$.
\end{lemma}

\begin{proof}
The composition $f\circ\eta$ is continuous and the differential $df\colon V\times E\to F$ 
is continuous and linear in the second argument,
thus $df\circ(\eta,\gamma)\in\LLeb pabF$ for $[\gamma]=\eta^\prime$, by 
Lemma \ref{lem:linear2-in-Lp}.
In other words, the function
	\begin{align}
		\zeta\colon[a,b]\to F,\quad \zeta(t):=f(\eta(a)) + \int_a^t df(\eta(s),\gamma(s)) \,ds
	\end{align}
is in $\ACp abF$.

We claim that for each continuous linear form $\alpha\in E'$, the composition
$\alpha\circ f\circ\eta$ is in $\ACp ab\R$ (hence almost everywhere differentiable) 
and that $(\alpha\circ f\circ\eta)^\prime(t) = \alpha(df(\eta(t),\gamma(t)))$ for
almost every $t\in[a,b]$. From this will follow that
	\begin{align*}
		\alpha(f(\eta(t))) 
		&= (\alpha\circ f\circ\eta)(a) + \int_a^t \alpha(df(\eta(s),\gamma(s))) \,ds\\
		&=\alpha\left(f(\eta(a)) + \int_a^t df(\eta(s),\gamma(s)) \,ds\right)
		=\alpha(\zeta(t)),
	\end{align*}
for each $\alpha\in E'$ and $t\in[a,b]$, therefore $f\circ\eta = \zeta\in\ACp abF$, as $E'$ separates points
on $E$.

To prove the claim, we may assume that $F=\R$ and we show that the composition $f\circ\eta\colon[a,b]\to\R$ is
absolutely continuous in the sense that for each $\varepsilon>0$ there exists some $\delta>0$
such that $\sum_{j=1}^n |f(\eta(b_j))-f(\eta(a_j))| < \varepsilon$ whenever
$a\leq a_1 < b_1\leq a_2 < b_2\leq \cdots \leq a_n < b_n \leq b$ with 
$\sum_{j=1}^n |b_j-a_j| < \delta$. 

Now, as $\eta([a,b])$ is a compact subset of the open subset $V$, there is some
open neighborhood $U\subseteq V$ of $\eta([a,b])$ and some continuous seminorm
$q$ on $E$ such that 
	\begin{align}\label{eq:C1-act-vecsp-help}
		|f(u)-f(\bar{u})| \leq q(u-\bar{u})
	\end{align}
for all $u$, $\bar{u} \in U$, by \cite[Lemma 1.60]{GlMR}. Given $\varepsilon>0$ there exists some $\delta>0$
such that $\sum_{j=1}^n |\sigma(b_j)-\sigma(a_j)| < \varepsilon$ whenever
$a\leq a_1 < b_1\leq a_2 < b_2\leq \cdots \leq a_n < b_n \leq b$ with 
$\sum_{j=1}^n |b_j-a_j| < \delta$, because the function
	\begin{align*}
		\sigma\colon[a,b]\to \R,\quad \sigma(t):=\int_a^t q(\gamma(s)) \,ds
	\end{align*}
is absolutely continuous. Therefore, we have
	\begin{align*}
		\sum_{j=1}^n |f(\eta(b_j))-f(\eta(a_j))|
		\leq \sum_{j=1}^n q(\eta(b_j) - \eta(a_j))
		&= \sum_{j=1}^n q\left(\int_{a_j}^{b_j} \gamma(s) \,ds\right)\\
		&\leq \sum_{j=1}^n \int_{a_j}^{b_j} q(\gamma(s)) \,ds 
		< \delta,
	\end{align*}
where we used (\ref{eq:C1-act-vecsp-help}) in the first step. Hence $f\circ\eta$ is 
absolutely continuous, thus, by \cite[VII. 4.14]{Elst}, there is some $\varphi\in\mathcal{L}^1([a,b])$
such that
	\begin{align*}
		f(\eta(t)) = f(\eta(a)) + \int_a^t \varphi(s) \,ds,
	\end{align*}
in other words 
	\begin{align*}
		f\circ\eta\in AC_{L^1} ([a,b],\R)\quad\mbox{ and }\quad
		(f\circ\eta)^\prime(t)=\varphi(t)\mbox{ for a.e. }t\in[a,b].
	\end{align*}
	
Now, we want to show that $\varphi(t)=df(\eta(t),\gamma(t))$ for almost every $t\in[a,b]$, that is,
$\varphi\in\mathcal{L}^p([a,b])$. To this end, we may assume that there exists a sequence
$(K_n)_{n\in\N}$ of compact subsets of $[a,b]$ such that for every
$n\in\N$ the restriction $\gamma|_{K_n}$ is continuous , 
$\lambda([a,b]\setminus\bigcup_{n\in\N}K_n)=0$ and $\gamma(t)=0$ for every
$t\notin\bigcup_{n\in\N}K_n$.

Each of the sets
	\begin{align*}
		L_n := \eta([a,b])\cup\bigcup_{m=1}^n \gamma(K_m)
	\end{align*}
is compact and metrizable, hence by \cite[Lemma 1.11]{GlMR}, there exists
a locally convex topology $\mathcal{T}_{X_n}$ on each vector subspace
	\begin{align*}
		X_n := \linspan(L_n),
	\end{align*}
which is metrizable, separable and coarser than the induced topology $\mathcal{O}_{X_n}$.
Then on each $X_n$, there is a countable family
$\Lambda_n$ of continuous (with respect to $\mathcal{T}_{X_n}$) linear functionals separating the points 
(see [Chapter II, Prop. 4]{Schwartz}).
Consequently, the countable family $\Lambda:=\bigcup_{n\in\N}\Lambda_n$ separates the points on the
vector space $X:=\bigcup_{n\in\N} X_n$, which enables to 
define a metrizable locally convex topology $\mathcal{T}_X$
coarser than the induced topology $\mathcal{O}_X$. 
On the other hand, each of the $m$-fold sums
	\begin{align*}
		L_{m,n} := [-m,m]L_n+\cdots+[-m,m]L_n
	\end{align*}
is compact (with respect to $\mathcal{T}_{X_n}$ and $\mathcal{O}_{X_n}$), 
and $X_n = \bigcup_{m\in\N}L_{m,n}$, thus
	\begin{align*}
		X = \bigcup_{m,n\in\N}L_{m,n}
	\end{align*}
is $\sigma$-compact.

Then, the space $X\times X\times \R$ has a locally convex metrizable $\sigma$-compact
topology $\mathcal{T}$, say, $X\times X\times \R = \bigcup_{n\in\N}C_n$.
Then $\mathcal{O}_{C_n} = \mathcal{T}_{C_n}$, where $\mathcal{O}_{C_n}$, $\mathcal{T}_{C_n}$
are the topologies on $C_n$ induced by $E\times E\times\R$ and $X\times X\times\R$, respectively.
Hence $V^{[1]}\cap C_n\in \mathcal{T}_{C_n}$ and $\mathcal{T}_{C_n}$ is compact and metrizable, hence
second countable. Therefore, $V^{[1]}\cap C_n$ is $\sigma$-compact (being locally compact with
countable base), that is, $V^{[1]}\cap C_n$ is a countable union of compact subsets, hence
so is $(V\cap X)^{[1]} = V^{[1]}\cap(X\times X\times \R) = \bigcup_{n\in\N}(V^{[1]}\cap C_n)$,
so we may write $(V\cap X)^{[1]} = \bigcup_{n\in\N}A_n$ with compact subsets $A_n$.

Next, we will construct a metrizable locally convex topology on $X$ such that
$\eta\in\ACp ab{V\cap X}$ with $\eta^\prime = [\gamma]\in\Leb pab{V\cap X}$ and 
such that $f^{[1]}|_{(V\cap X)^{[1]}}$ remains continuous. 
To this end, fix some $(x,y,t)\in (V\cap X)^{[1]}$. Then there
exists $n\in\N$ such that $(x,y,t)\in A_n$. Further, for every $k\in\N$
there exists a continuous seminorm $q_{n,k}$ on $E\times E\times\R$ such that
	\begin{align*}
		|f^{[1]}(x,y,t)-f^{[1]}(v,w,s)|<\frac{1}{k} \quad\quad \forall (v,w,s)\in B_1^{q_{n,k}}(x,y,t)
	\end{align*}
(see Lemma \ref{lem:compact-continuity}).
Consequently, there is a continuous seminorm $\pi_{n,k}$ on $E$ and $\delta>0$
such that 
	\begin{align*}
		|f^{[1]}(x,y,t)&-f^{[1]}(v,w,s)|<\frac{1}{k} \\
		&\forall (v,w,s)\in B_1^{\pi_{n,k}}(x)\times B_1^{\pi_{n,k}}(y)\times ]t-\delta,t+\delta[
		\subseteq B_1^{q_{n,k}}(x,y,t).
	\end{align*}

We endow $X$ with the metrizable locally convex topology $\mathcal{T}$ defined by the countable
family $\{\pi_{n,k}|_X : n,k\in\N\}$. This topology is coarser than the induced topology,
hence $\eta\colon[a,b]\to V\cap X$ remains continuous and $\eta(t)-\eta(a)=\int_a^t \gamma(s) \,ds$
is the weak integral of $\gamma$ in $X$ for every $t\in[a,b]$. 
To see this, let $\alpha$ be a continuous linear functional on $(X,\mathcal{T})$.
Then $\alpha$ is continuous with respect to the induced topology on $X$ (which is finer than $\mathcal{T}$)
hence there is some continuous linear extension $\mathcal{A}\in E'$ of $\alpha$.
Thus
	\begin{align*}
		\alpha(\eta(t)-\eta(a)) = \mathcal{A}(\eta(t)-\eta(a))
		=\int_a^t \mathcal{A}(\gamma(s)) \,ds = \int_a^t \alpha(\gamma(s)) \,ds.
	\end{align*}
Therefore, $\eta\in\ACp ab{V\cap X}$ with $\eta^\prime=[\gamma]$ and, by the
construction of the topology, the map $f^{[1]}$ is continuous
in every $(x,y,t)\in(V\cap X)^{[1]}$ with respect to the obtained topology on $X\times X\times \R$.
As $\mathcal{T}$ is metrizable, the map $\eta\colon[a,b]\to V\cap X$ is differentiable in almost
every $t\in[a,b]$ with $\eta^\prime(t)=\gamma(t)$ (see Proposition \ref{prop:weak-integral-diff}),
so in every such $t$ we have
	\begin{align*}
		\frac{1}{h}(f(\eta(t+h) - f(\eta(t)))
		&=\frac{1}{h}(f(\eta(t)+\frac{\eta(t+h)-\eta(t)}{h})-f(\eta(t)))\\
		&=f^{[1]}(\eta(t),\frac{\eta(t+h)-\eta(t)}{h},h)
		\to df(\eta(t),\gamma(t))
	\end{align*}
as $h\to 0$.
That means, for almost every $t\in[a,b]$ we have
	\begin{align}\label{eq:AC-vecsp-chain-rule}
		\varphi(t)=(f\circ\eta)^\prime(t)=df(\eta(t),\gamma(t)),
	\end{align}
whence $\varphi\in\mathcal{L}^p([a,b])$ and $f\circ\eta\in \ACp ab\R$.
\end{proof}

Also the following (corresponding to \cite[Lemma 3.28]{GlMR}) holds:

\begin{proposition}\label{prop:push-forward-AC-vecsp}
Let $E$, $F$ be sequentially complete locally convex spaces, let $V\subseteq E$ be an open subset
and $p\in[1,\infty]$.
If $f\colon V\to F$ is a $C^{k+2}$-function (for $k\in\N\cup\{0,\infty\}$), then the map
	\begin{align*}
		\ACp abf\colon\ACp abV \to \ACp abF,\quad
		\eta\mapsto f\circ\eta
	\end{align*}
is $C^k$.
\end{proposition}

\begin{proof}
The map $\ACp abf$ is defined by Lemma \ref{lem:C1-maps-act-on-AC-vecsp}, by 
definition of the topology on $\ACp abF$ (see Definition \ref{def:abs-cont-vector-space}), 
$\ACp abf$ will be $C^k$ if each of the components of
	\begin{align}\label{eq:push-forward-AC-vecsp-help}
		\ACp abV\to F\times \Leb pabF,\quad
		\eta\mapsto (f(\eta(a)), (f\circ\eta)^\prime)
	\end{align}
is $C^k$.
The first component 
	\begin{align*}
		\ACp abV\to F,\quad \eta\mapsto (f\circ\pr_1\circ\Phi)(\eta)=f(\eta(a))
	\end{align*}
is indeed $C^k$, where $\Phi$ is as in Definition \ref{def:abs-cont-vector-space}
and $\pr_1$ is the projection on the first component.
Further, for $\eta^\prime = [\gamma]\in\Leb pabE$ we have $(f\circ\eta)^\prime = [df\circ(\eta,\gamma)]$
by (\ref{eq:AC-vecsp-chain-rule}) and 
	\begin{align*}
		C([a,b],V)\times \Leb pabE\to\Leb pabF,\quad
		(\eta,[\gamma])\mapsto [df\circ(\eta,\gamma)]
	\end{align*}
is $C^k$, the derivative $df\colon V\times E\to F$ being $C^{k+1}$ and linear in the second
argument (see Proposition \ref{prop:double-push-forward-Lp}). Hence, the second component
of (\ref{eq:push-forward-AC-vecsp-help}) is $C^k$, as required.
\end{proof}

\begin{remark}\label{rem:AC-vecsp-prod}
Since any continuous linear function $f\colon E\to F$ is smooth, we conclude from
Proposition \ref{prop:push-forward-AC-vecsp}
that
	\begin{align*}
		\ACp ab{E\times F} \cong \ACp abE \times \ACp abF
	\end{align*}
as locally convex spaces (proceeding as in Remark \ref{rem:Lp-product}).
\end{remark}

The next result will be helpful.

\begin{lemma}\label{lem:AC-vecsp-locality-axiom}
Let $E$ be a sequentially complete locally convex space, $p\in[1,\infty]$ and $a=t_0<t_1<\ldots<t_n=b$. 
Then the function
	\begin{align}
		\Psi\colon \ACp abE \to \prod_{j=1}^n \ACp {t_{j-1}}{t_j}E,\quad
		\eta\mapsto (\eta|_{[t_{j-1},t_j]})_{j=1,\ldots,n}
	\end{align}
is a linear topological embedding with closed image.
\end{lemma}

\begin{proof}
Clearly, for $\eta\in\ACp abE$ with $\eta^\prime = [\gamma]$ and every $j\in\{1,\ldots,n\}$ we have 
$\eta|_{[t_{j-1},t_j]}\in\ACp {t_{j-1}}{t_j}E$ with $(\eta|_{[t_{j-1},t_j]})^\prime = \left[\gamma|_{[t_{j-1},t_j]}\right]$,
by Lemma \ref{lem:Lp-locality-axiom}, that is, the function $\Psi$ is defined. 
Also the linearity and injectivity are clear. 

We show that each of the components
	\begin{align*}
		\ACp abE\to \ACp {t_{j-1}}{t_j}E,\quad
		\eta\mapsto \eta|_{[t_{j-1},t_j]}
	\end{align*}
is continuous, which will be the case if each
	\begin{align*}
		\ACp abE\to E\times\Leb p{t_{j-1}}{t_j}E,\quad
		\eta\mapsto \left(\eta(t_{j-1}),\left[\gamma|_{[t_{j-1},t_j]}\right]\right)
	\end{align*}
is continuous (using the isomorphism as in Definition \ref{def:abs-cont-vector-space}).
But the first component is the continuous evaluation map on $\ACp abE$, see
Remark \ref{rem:eval-smooth-AC-vecsp}, and the second component is a composition
of the continuous maps $\ACp abE\to\Leb pabE, \eta\mapsto [\gamma]$
and $\Leb pabE\to \Leb p{t_{j-1}}{t_j}E, [\gamma]\mapsto \left[\gamma|_{[t_{j-1},t_j]}\right]$,
see Definition \ref{def:abs-cont-vector-space} and Lemma \ref{lem:Lp-locality-axiom}.
Therefore, $\Psi$ is continuous.

Note that $(\eta_1,\ldots,\eta_n)\in\im(\Psi)\subseteq \prod_{j=1}^n \ACp {t_{j-1}}{t_j}E$ if 
$\eta_j(t_j) = \eta_{j+1}(t_j)$ for all $j\in\{1,\ldots,n-1\}$, thus the map
	\begin{align*}
		\Gamma(\eta_1,\ldots,\eta_n)\colon [a,b]\to E, \quad t\mapsto \eta_j(t) \mbox{ for } t\in[t_{j-1},t_j]
	\end{align*}
is continuous and it is easy to show that $\Gamma(\eta_1,\ldots,\eta_n)\in\ACp abE$
and that
	\begin{align*}
		\Gamma\colon\im(\Psi)\to\ACp abE
	\end{align*}
is the inverse of $\Psi|^{\im(\Psi)}$. The continuity of $\Gamma$ follows from the continuity of
	\begin{align*}
		\im(\Psi)&\to E\times \Leb pabE,\\
		(\eta_1,\ldots,\eta_n)&\mapsto (\eta_1(a), \Gamma(\eta_1,\ldots,\eta_n)^\prime)
		=(\eta_1(a), \Gamma_E^{-1}(\eta_1^\prime,\ldots,\eta_n^\prime)),
	\end{align*}
where $\Gamma_E$ is the isomorphism from Lemma \ref{lem:Lp-locality-axiom}.
Hence $\Psi$ is a topological embedding.

Finally, let $(\eta_{1,\alpha},\ldots,\eta_{n,\alpha})_{\alpha\in A}$ be a convergent net in $\im(\Psi)$
with limit point $(\eta_1,\ldots,\eta_n)\in\prod_{j=1}^n \ACp {t_{j-1}}{t_j}E$, thus
for every $j\in\{1,\ldots,n-1\}$ we have
	\begin{align*}
		\eta_j(t_j) = \lim_{\alpha\in A}\eta_{j,\alpha}(t_j) = 
		\lim_{\alpha\in A}\eta_{j+1,\alpha}(t_j) = \eta_{j+1}(t_j),
	\end{align*}
therefore $(\eta_1,\ldots,\eta_n)\in\im(\Psi)$.
\end{proof}

\begin{remark}\label{rem:restriction-abs-cont-vecsp}
Lemma \ref{lem:AC-vecsp-locality-axiom} immediately implies that whenever $\eta\in C([a,b],E)$
and $\eta|_{[t_{j-1},t_j]}\in \ACp {t_{j-1}}{t_j}E$ for some $\abpart$, we have $\eta\in\ACp abE$.
\end{remark}

\begin{lemma}\label{lem:affine-linear-AC-vecsp}
Let $E$ be a sequentially complete locally convex space, $p\in[1,\infty]$.
Let $\eta\in\ACp abE$ and for $a\leq \alpha<\beta\leq b$ define
	\begin{align*}
		f\colon[c,d]\to[a,b],\quad f(t):=\alpha + \frac{t-c}{d-c}(\beta-\alpha).
	\end{align*}
Then $\eta\circ f\in\ACp cdE$ and
	\begin{align*}
		(\eta\circ f)^\prime = \frac{\beta-\alpha}{d-c}[\gamma\circ f],
	\end{align*}
where $[\gamma]=\eta^\prime$.

Furthermore, the function
	\begin{align*}
		AC_{L^p}(f,E)\colon\ACp abE\to\ACp cdE,\quad \eta\mapsto \eta\circ f
	\end{align*}
is continuous and linear.
\end{lemma}

\begin{proof}
We know that $\gamma\circ f\in\LLeb pcdE$ (Lemma \ref{lem:affine-linear-Lp}) and for $t\in[c,d]$
we have
	\begin{align*}
		\eta(f(t))-\eta(f(c)) 
		= \int_{f(c)}^{f(t)} \gamma(s) \,ds,
	\end{align*}
since $[\gamma]=\eta^\prime$. Then for any continuous linear functional $\mathcal{A}$ on $E$
we have
	\begin{align*}
		\int_{f(c)}^{f(t)}\mathcal{A}(\gamma(s)) \,ds 
		= \frac{\beta-\alpha}{d-c}\int_c^t \mathcal{A}(\gamma(f(s))) \,ds
	\end{align*}
(see \cite[19.4 Satz]{B92}), whence
	\begin{align*}
		\eta(f(t))-\eta(f(c)) = \frac{\beta-\alpha}{d-c}\int_c^t \gamma(f(s)) \,ds,
	\end{align*}
in other words, $\eta\circ f\in \ACp cdE$ with $(\eta\circ f)^\prime = \frac{\beta-\alpha}{d-c}[\gamma\circ f]$.

To prove the continuity of the linear function $AC_{L^p}(f,E)$, we show that
	\begin{align*}
		\ACp abE\to E\times \Leb pcdE,\quad \eta\mapsto (\eta(f(c)), (\eta\circ f)^\prime)
	\end{align*}
is continuous (where we used the isomorphism from Definition \ref{def:abs-cont-vector-space}).
The first component
	\begin{align*}
		\ACp abE\to E,\quad \eta\mapsto \ev_{f(c)}(\eta)
	\end{align*}
is continuous, by Remark \ref{rem:eval-smooth-AC-vecsp}. Further, the map
	\begin{align*}
		\Psi\colon \Leb pabE\to \Leb pcdE,\quad [\gamma]\mapsto \frac{\beta-\alpha}{d-c}[\gamma\circ f]
	\end{align*}
is continuous, by Lemma \ref{lem:affine-linear-Lp}, hence the second component
	\begin{align*}
		\ACp abE\to \Leb pcdE,\quad \eta\mapsto\Psi(\eta^\prime)=(\eta\circ f)^\prime
	\end{align*}
is continuous, and the proof is finished.
\end{proof}

\subsection{Manifold-valued \emph{AC}-functions}

Let $M$ be a smooth manifold modelled on a locally convex space $E$, that is,
$M$ be a Hausdorff topological space together with an atlas of charts 
$\varphi\colon U_{\varphi}\to V_{\varphi}$ (homeomorphisms between open subsets of $M$ and $E$)
such that the transition maps $\psi\circ\varphi^{-1}$ are smooth functions
(as in \cite{GlHRestr}). The definition of the tangent space $T_xM$, the tangent manifold $TM$,
differentiable functions between manifolds and tangent maps between tangent spaces
are defined in the usual way.

The properties of the spaces $\ACp abE$, proved in the preceeding, enable us to
define spaces of absolutely continuous functions with values in manifolds $M$
modelled on sequentially complete locally convex spaces.

The corresponding definition can be found in \cite[Definition 3.20]{GlMR}.

\begin{definition}\label{def:abs-cont-mfd}
Let $M$ be a smooth manifold modelled on a sequentially complete locally convex space $E$.
For $p\in[1,\infty]$, denote by $\ACp abM$ the space of continuous functions $\eta\colon [a,b]\to M$
such that there exists some partition $\abpart$ with
	\begin{align*}
		\varphi_j\circ\eta|_{[t_{j-1},t_j]}\in \ACp {t_{j-1}}{t_j}E
	\end{align*}
for some charts $\varphi_j\colon U_j\to V_j$ such that $\eta([t_{j-1},t_j])\subseteq U_j$ for
$j=1,\ldots,n$.
\end{definition}

As in \cite[Lemma 3.21]{GlMR}, the construction of $\ACp abM$ does not depend
on the choice of the partition $\abpart$ or of the charts $\varphi_j$.

\begin{lemma}\label{lem:abs-cont-mfd-each-chart}
Let $\eta\in\ACp abM$, let $[\alpha,\beta]\subseteq[a,b]$ and $\varphi\colon U\to V$ be any
chart for $M$ such that $\eta([\alpha,\beta])\subseteq U$. Then
	\begin{align*}
		\varphi\circ\eta|_{[\alpha,\beta]}\in\ACp \alpha\beta E.
	\end{align*}
\end{lemma}

\begin{proof}
We have $\alpha\in[t_k,t_{k+1}]$ and $\beta\in[t_{l-1},t_l]$ for some $k,l$, for simplicity
we may assume $\alpha=t_k$, $\beta=t_l$. For $j\in\{k+1,\ldots,l\}$ we have
	\begin{align*}
		\varphi\circ\eta|_{[t_{j-1},t_j]} = (\varphi\circ\varphi_j^{-1})\circ(\varphi_j\circ\eta|_{[t_{j-1},t_j]}).
	\end{align*}
Since $\varphi\circ\varphi_j^{-1}$ is a smooth function and 
$\varphi_j\circ\eta|_{[t_{j-1},t_j]} \in \ACp {t_{j-1}}{t_j}E$, the above composition is in $\ACp {t_{j-1}}{t_j}E$
by Lemma \ref{lem:C1-maps-act-on-AC-vecsp}. From Remark \ref{rem:restriction-abs-cont-vecsp} follows
$\varphi\circ\eta|_{[\alpha,\beta]}\in\ACp \alpha\beta E$.
\end{proof}

\begin{remark}\label{rem:AC-vecsp-mfd}
From the above lemma follows that if the smooth manifold $M$ is a locally convex space (with global
chart $\id_M$), then 
the space $\ACp abM$ defined in Definition $\ref{def:abs-cont-mfd}$ is the same as the space defined
in Definition \ref{def:abs-cont-vector-space}.
\end{remark}

The next results correspond to \cite[Lemma 3.24]{GlMR} and \cite[Lemma 3.30]{GlMR}.

\begin{lemma}\label{lem:C1-maps-act-on-AC-mfd}
Let $M$, $N$ be smooth manifolds modelled on sequentially complete locally convex spaces
$E$ and $F$, respectively. If $f\colon M\to N$ is a $C^1$-map, then $f\circ\eta\in\ACp abN$
for each $\eta\in\ACp abM$ and $p\in[1,\infty]$.
\end{lemma}

\begin{proof}
Consider a partition $\abpart$ and charts $\varphi_j\colon U_j\to V_j$ for $M$ such that
$\eta([t_{j-1},t_j])\subseteq U_j$ and $\varphi_j\circ\eta|_{[t_{j-1},t_j]}\in\ACp {t_{j-1}}{t_j}E$
for each $j\in\{1,\ldots,n\}$. Since $f\circ\eta|_{[t_{j-1},t_j]}$ is continuous, we find a partition
$t_{j-1}=s_0<s_1<\cdots<s_m=t_j$ and charts $\psi_i\colon P_i\to Q_i$ for $N$ such that
$f(\eta([s_{i-1},s_i]))\subseteq P_i$ for each $i\in\{1,\ldots,m\}$.
Then
	\begin{align*}
		\psi_i\circ f\circ\eta|_{[s_{i-1},s_i]} 
		= (\psi_i\circ f\circ\varphi_j^{-1})\circ(\varphi_j\circ\eta|_{[s_{i-1},s_i]}) \in\ACp {s_{i-1}}{s_i}F, 
	\end{align*}
by Remark \ref{rem:restriction-abs-cont-vecsp} and Lemma \ref{lem:C1-maps-act-on-AC-vecsp}.
Hence $f\circ\eta|_{[t_{j-1},t_j]}\in\ACp {t_{j-1}}{t_j}N$ for each $j\in\{1,\ldots,n\}$, whence
$f\circ\eta\in\ACp abN$.
\end{proof}

\begin{lemma}\label{lem:double-push-forward-AC-vecsp-mfd}
Let $E_1$, $E_2$ and $F$ be sequentially complete locally convex spaces. Let
$M$ be a smooth manifold modelled on $E_1$ and $V\subseteq E_2$ be an open subset.
If $f\colon M\times V\to F$ is a $C^{k+2}$-map and $\zeta\in\ACp abM$ for $p\in[1,\infty]$, then
	\begin{align}
		\ACp abV \to \ACp abF,\quad
		\eta\mapsto f\circ(\zeta,\eta)
	\end{align}
is a $C^k$-map.
\end{lemma}

\begin{proof}
Since $(\zeta,\eta)\in\ACp ab{M\times V}$, the above map is defined
by Lemma \ref{lem:C1-maps-act-on-AC-mfd}; it will be $C^k$ if for a partition $\abpart$ for $\zeta$
(as in Definition \ref{def:abs-cont-mfd}) the function
	\begin{align*}
		\ACp abV&\to\prod_{j=1}^n \ACp {t_{j-1}}{t_j}F,\\
		\eta&\mapsto \left( f\circ(\zeta,\eta)|_{[t_{j-1},t_j]}\right)_{j=1,\ldots,n}
	\end{align*}
is $C^k$ (where we use the topological embedding with 
closed image on the space $\ACp abF$ as in Lemma \ref{lem:AC-vecsp-locality-axiom}).
This will hold if every component
	\begin{align}\label{eq:double-push-forward-help}
		\ACp abV\to\ACp {t_{j-1}}{t_j}F,\quad
		\eta\mapsto f\circ(\zeta,\eta)|_{[t_{j-1},t_j]}
	\end{align}
is $C^k$.

Now, given charts $\varphi\colon U_j\to V_j$ for $M$ with $\zeta([t_{j-1},t_j])\subseteq U_j$,
for every $j\in\{1,\ldots,n\}$, the function
	\begin{align*}
		\ACp abV\to\ACp{t_{j-1}}{t_j}{V_j\times V},\quad
		\eta\mapsto\left(\varphi_j\circ\zeta|_{[t_{j-1},t_j]} , \eta|_{[t_{j-1},t_j]}\right)
	\end{align*}
is smooth by Lemma \ref{lem:AC-vecsp-locality-axiom} (identifying $\ACp{t_{j-1}}{t_j}{V_j\times V}$
with the product $\ACp {t_{j-1}}{t_j}{V_j}\times\ACp {t_{j-1}}{t_j}V$, see Remark \ref{rem:AC-vecsp-prod} ). 
As the composition $f\circ(\varphi_j^{-1}\times\id_V)\colon V_j\times V\to F$ is $C^{k+2}$, by
Proposition \ref{prop:push-forward-AC-vecsp} the function
	\begin{align*}
		\ACp abV&\to\ACp {t_{j-1}}{t_j}F, \\
		\eta&\mapsto \left( f\circ(\varphi_j^{-1}\times\id_V)
		\circ(\varphi_j\circ\zeta|_{[t_{j-1},t_j]} , \eta|_{[t_{j-1},t_j]}\right)\\
		&= f\circ(\zeta,\eta)|_{[t_{j-1},t_j]}
	\end{align*}
is $C^k$. Therefore, the function in (\ref{eq:double-push-forward-help}) is $C^k$ and the
proof is finished.
\end{proof}

\subsection{Lie group-valued \emph{AC}-functions}

We consider smooth Lie groups $G$ modelled on locally convex spaces $E$, that is, $G$ is
a group endowed with a 
smooth manifold structure modelled on $E$ such that the group multiplication $m_G\colon G\times G\to G$
and the inversion $j_G\colon G$ are smooth functions. We will always denote the identity element of $G$
by $e_G$ and the Lie group by $\mathfrak{g}:=T_{e_G}G\cong E$.

Similar to \cite[Proposition 4.2]{GlMR}, we endow $\ACp abG$ with a Lie group
structure for any Lie group $G$ modelled on a sequentially complete locally
convex space, after recalling the following fact.

\begin{remark}\label{rem:Lie-group-local}
Let $G$ be a group, $U\subseteq G$ be a symmetric subset containing the identity element of $G$.
Assume that $U$ is endowed with a smooth manifold structure modelled on a locally convex space $E$
such that the inversion $U\to U, x\to x$ on $U$ is smooth, the subset $U_m:=\{(x,y)\in U\times U : xy\in U\}$
is open in $U \times U$ and the multiplication $U_m\to U, (x,y)\mapsto xy$ is smooth on $U_m$.
Further, assume that for each $g\in G$, there exists an open identity neighborhood $W\subseteq U$
such that $gWg^{-1}\subseteq U$ and $W\to U, x\mapsto gxg^{-1}$ is smooth.
Then $G$ can be endowed with a unique smooth manifold structure modelled on $E$
such that $G$ becomes a smooth Lie group and $U$ with the given manifold structure becomes
an open smooth submanifold.
\end{remark}

\begin{proposition}\label{prop:abs-cont-Lie-gr}
Let $G$ be a smooth Lie group modelled on a sequentially complete locally convex space $E$,
let $p\in[1,\infty]$. Then
there exists a unique Lie group structure on $\ACp abG$ such that for each open symmetric 
$e_G$-neighborhood $U\subseteq G$ the subset $\ACp abU$ is open in $\ACp abG$ and
such that each
	\begin{align*}
		\ACp ab\varphi\colon\ACp abU\to\ACp abV,\quad
		\eta\mapsto\varphi\circ\eta
	\end{align*}
is a smooth diffeomorphism for every chart $\varphi\colon U\to V$ for $G$.
\end{proposition}

\begin{proof}
\emph{Step 1: $\ACp abG$ is a group.}

As $m_G$ and $j_G$
are smooth, we have
$m_G\circ (\eta,\xi), j_G\circ\eta\in\ACp abG$ for all $\eta$, $\xi\in\ACp abG$, 
by Lemma \ref{lem:C1-maps-act-on-AC-mfd}
(identifying $\ACp ab{G\times G}$ with $\ACp abG\times\ACp abG$). Then
$\tilde{G}:=\ACp abG$ is a group with multiplication
	\begin{align*}
		m_{\tilde{G}}:=\ACp ab{m_G}\colon \tilde{G}\times\tilde{G}\to\tilde{G},\quad
		(\eta,\xi)\mapsto m_G\circ(\eta,\xi)=:\eta\cdot\xi,
	\end{align*}
inversion
	\begin{align*}
		j_{\tilde{G}}:=\ACp ab{j_G}\colon \tilde{G}\to\tilde{G},\quad 
		\eta\mapsto j_G\circ\eta=:\eta^{-1}
	\end{align*}
and identity element $e_{\tilde{G}}\colon t\mapsto e_G$.

\emph{Step 2: Existence of a Lie group structure on $\ACp abG$.}

Consider an open symmetric $e_G$-neighborhood $U\subseteq G$ and a chart $\varphi\colon U\to V$.
As $\tilde{V}:=\ACp abV$ is open in $\ACp abE$ (see Remark \ref{rem:open-in-AC-vecsp}), we endow 
the symmetric subset $\tilde{U}:=\ACp abU := \{\eta\in\ACp abG : \eta([a,b])\subseteq U\}$ with the
$C^\infty$-manifold structure turning the bijection
	\begin{align*}
		\tilde{\varphi}:=\ACp ab\varphi\colon \tilde{U}\to\tilde{V},\quad \eta\to\varphi\circ\eta
	\end{align*}
into a global chart (the map is defined by Lemma \ref{lem:C1-maps-act-on-AC-mfd}). 
Obviously, $e_{\tilde{G}}\in\tilde{U}$.

Further, by Lemma \ref{lem:C1-maps-act-on-AC-vecsp}, the function
	\begin{align*}
		\ACp ab{\varphi\circ j_G|_U \circ \varphi^{-1}}\colon \tilde{V}\to\tilde{V},
		\quad \eta\mapsto (\varphi\circ j_G|_U \circ \varphi^{-1})\circ\eta
	\end{align*}
is smooth. Thus, writing
	\begin{align*}
		\tilde{U}\to\tilde{U},
		\quad \eta&\mapsto (\tilde{\varphi}^{-1}\circ \ACp ab{\varphi\circ j_G|_U \circ \varphi^{-1}}\circ\tilde{\varphi})(\eta)\\
		&=\varphi^{-1}\circ\varphi\circ j_G|_U \circ \varphi^{-1}\circ\varphi\circ\eta\\
		&=j_G\circ\eta,
	\end{align*}
we see that the inversion on $\tilde{U}$ is smooth.

Now, consider the open subset $U_m:= \{(x,y)\in U\times U : xy\in U\}$ of $U\times U$.
As $V_m := (\varphi\times\varphi)(U_m)$ is open in $E\times E$, the set 
$\tilde{V}_m := \ACp ab{V_m}$ is open in $\ACp abE\times\ACp abE$, whence
$\tilde{U}_m := (\tilde{\varphi}^{-1}\times\tilde{\varphi}^{-1})(\tilde{V}_m)$ is open in $\tilde{U}\times\tilde{U}$.
Again, by Lemma \ref{lem:C1-maps-act-on-AC-vecsp}, the function
	\begin{align*}
		\ACp ab{\varphi\circ m_G\circ(\varphi^{-1}\times\varphi^{-1})|_{V_m}}\colon\tilde{V}_m&\to \tilde{V},\\
		\eta&\mapsto (\varphi\circ m_G\circ(\varphi^{-1}\times\varphi^{-1})|_{V_m})\circ\eta
	\end{align*}
is smooth. Therefore
	\begin{align*}
		\tilde{U}_m&\to\tilde{U},\\
		(\eta,\xi)&\mapsto (\tilde{\varphi}^{-1}\circ\ACp ab{\varphi\circ m_G\circ(\varphi^{-1}\times\varphi^{-1})|_{V_m}}
		\circ(\tilde{\varphi}\times\tilde{\varphi}))(\eta,\xi)\\
		&=\varphi^{-1}\circ\varphi\circ m_G\circ(\varphi^{-1}\times\varphi^{-1})|_{V_m}\circ(\varphi\times\varphi)\circ(\eta,\xi)\\
		&=m_G\circ(\eta,\xi),
	\end{align*}
which is the multiplication on $\tilde{U}_m$, is smooth.

Finally, fix some $\eta\in\tilde{G}$ and write $K:=\im(\eta)\subseteq G$. As the function 
	\begin{align*}
		h\colon G\times G\to G, (x,y)\mapsto xyx^{-1}
	\end{align*}
is smooth and $h({K}\times\{e_G\}) = \{e_G\}\subseteq U$,
the compact set $K\times\{e_G\}$ is a subset of the open set $h^{-1}(U)\subseteq G\times G$. By
the Wallace Lemma, there are open subsets $W_K$, $W$ of $G$ such that
$K\times\{e_G\}\subseteq W_K\times W\subseteq h^{-1}(U)$.
We may assume $W\subseteq U$, then we see that $\tilde{W}:=\ACp abW$ is open in $\tilde{U}$
and for each $\xi\in\tilde{W}$ we have
	\begin{align*}
		\eta\cdot\xi\cdot\eta^{-1}=h\circ(\eta,\xi)\in\tilde{U}
	\end{align*}
by Lemma \ref{lem:C1-maps-act-on-AC-mfd}. Using Lemma \ref{lem:double-push-forward-AC-vecsp-mfd},
we see that the function
	\begin{align*}
		\ACp ab{\varphi(W)}&\to\tilde{V},\\
		\xi&\mapsto (\varphi\circ h\circ (\id_{W_K}\times\varphi^{-1}|_{\varphi(W)})\circ(\eta,\xi)
	\end{align*}
is smooth, whence
	\begin{align*}
		\tilde{W}\to\tilde{U},\quad 
		\xi&\mapsto
		\varphi^{-1}\circ (\varphi\circ h\circ (\id_{W_K}\times\varphi^{-1}|_{\varphi(W)})\circ (\eta, \varphi\circ\xi)\\
		&=h\circ(\eta,\xi)\\
		&=\eta\cdot\xi\cdot\eta^{-1}
	\end{align*}
is smooth. 

Consequently, by Remark \ref{rem:Lie-group-local}, there exists a unique Lie group structure on $\tilde{G}$
turning $\tilde{U}$ into a smooth open submanifold and $\tilde{\varphi}$ into a $\tilde{G}$-chart around $e_{\tilde{G}}$.

\emph{Step 3: Uniqueness of the Lie group structure on $\ACp abG$.}

Let $U'\subseteq G$ be an open symmetric $e_G$-neighborhood and $\varphi'\colon U'\to V'$
be a $G$-chart around $e_G$. Denote by $\tilde{G'}$ the group $\ACp abG$ endowed with the
Lie group structure turning $\tilde{U'}:=\ACp ab{U'}$ into an open submanifold and 
$\tilde{\varphi'}\colon \tilde{U'}\to \ACp ab{V'}$ into a chart (constructed as in Step $2$). We 
show that both identity maps $\id\colon \tilde{G'}\to\tilde{G}$ and $\id\colon\tilde{G}\to\tilde{G'}$ are
continuous, that is, both Lie group structures coincide.

The subset $U'\cap U$ is open in $U'$, therefore $\varphi'(U'\cap U)$ is open in $V'$, thus
$\ACp ab{\varphi'(U'\cap U)}$ is open in $\ACp ab{V'}$, and consequently 
$\tilde{U'}\cap\tilde{U}=\tilde{\varphi'}^{-1}(\ACp ab{\varphi'(U'\cap U)})$ is open in $\tilde{G'}$.
Writing
	\begin{align*}
		\id_{\tilde{U'}\cap\tilde{U}} 
		= \tilde{\varphi}^{-1}\circ\ACp ab{\varphi\circ\varphi'^{-1}|_{U'\cap U}}\circ\tilde{\varphi'}|_{\tilde{U'}\cap\tilde{U}}
		\colon \tilde{U'}\cap\tilde{U}\to\tilde{G}
	\end{align*}
and using Lemma \ref{lem:C1-maps-act-on-AC-vecsp}, we see that $\id\colon\tilde{G'}\to\tilde{G}$ is
continuous on the open identity neighborhood $\tilde{U'}\cap\tilde{U}$, hence continuous. In the same
way, we show that also $\id\colon\tilde{G}\to\tilde{G'}$ is continuous, as required.
\end{proof}

See \cite[Remark 4.3]{GlMR} for the next result.

\begin{lemma}\label{lem:incl-AC-C-Lie-gr}
The inclusion map
	\begin{align*}
		\incl\colon \ACp abG\to C([a,b],G),\quad \eta\mapsto\eta
	\end{align*}
is a smooth homomorphism.
\end{lemma}

\begin{proof}
Let $U\subseteq G$ be an open identity neighborhood, $\varphi\colon U\to V$ be
a chart for $G$. Then $C([a,b],\varphi)\colon C([a,b],U)\to C([a,b],V), \eta\mapsto \varphi\circ\eta$
is a chart for $C([a,b],G)$ and $\ACp ab{\varphi}\colon \ACp abU\to \ACp abV, \eta\mapsto \varphi\circ\eta$
is a chart for $\ACp abG$.
The function
	\begin{align*}
		\ACp abV&\to C([a,b],V),\\ 
		\eta&\mapsto \left(C([a,b],\varphi)\circ \incl\circ \ACp ab{\varphi}^{-1}\right)(\eta)=\eta
	\end{align*}
is smooth, being a restriction of the smooth inclusion map from Lemma \ref{lem:AC-vecsp-top-embedding}.
Hence the group homomorphism $\incl$ is smooth.
\end{proof}

Lemmas \ref{lem:eval-smooth-AC-Liegr} - \ref{lem:AC-Lie-gr-locality-axiom}
can be found in \cite[Lemma 4.8]{GlMR}.

\begin{lemma}\label{lem:eval-smooth-AC-Liegr}
For any $\alpha\in[a,b]$, the evaluation map
	\begin{align*}
		\ev_{\alpha}\colon \ACp abG\to G,\quad \eta\mapsto \eta(\alpha)
	\end{align*}
is a smooth homomorphism.
\end{lemma}

\begin{proof}
The function is a composition of the smooth inclusion map from Lemma \ref{lem:incl-AC-C-Lie-gr}
and the smooth evaluation map on $C([a,b],G)$, hence smooth.
\end{proof}

\begin{lemma}\label{lem:affine-linear-AC-liegr}
Let $\eta\in\ACp abG$ and for $a\leq \alpha<\beta\leq b$ define
	\begin{align*}
		f\colon[c,d]\to[a,b],\quad f(t):=\alpha + \frac{t-c}{d-c}(\beta-\alpha).
	\end{align*}
Then $\eta\circ f\in\ACp cdG$.

Furthermore, the function
	\begin{align*}
		AC_{L^p}(f,G)\colon \ACp abG\to\ACp cdG,\quad \eta\mapsto \eta\circ f
	\end{align*}
is a smooth homomorphism.
\end{lemma}

\begin{proof}
As $\eta\circ f$ is a continuous curve, there exists a partition $c=t_0<t_1<\ldots<t_n=d$
and for every $j\in\{1,\ldots,n\}$ there is a chart $\varphi_j\colon U_j\to V_j$ for $G$
such that $\eta(f([t_{j-1},t_j]))\subseteq U_j$. But $f([t_{j-1},t_j]) = [f(t_{j-1}),f(t_j)]$
is an interval and from Lemma \ref{lem:abs-cont-mfd-each-chart} follows that
	\begin{align*}
		\varphi_j\circ\eta|_{[f(t_{j-1}),f(t_j)]}\in\ACp {f(t_{j-1})}{f(t_j)}{V_j}.
	\end{align*}
By Lemma \ref{lem:affine-linear-AC-vecsp}, we have 
	\begin{align*}
		\varphi_j\circ\eta\circ f|_{[t_{j-1},t_j]}\in\ACp {t_{j-1}}{t_j}{V_j},
	\end{align*}
that is, $\eta\circ f\in \ACp cdG$.

Finally, for any open identity neighborhood $U\subseteq G$ and any chart $\varphi\colon U\to V$
for $G$ the function
	\begin{align*}
		\ACp abV&\to\ACp cdV,\\
		\zeta&\mapsto (\ACp cd{\varphi}\circ AC_{L^p}(f,G)\circ \ACp ab{\varphi}^{-1})(\zeta)=\zeta\circ f
	\end{align*}
is smooth (see Lemma \ref{lem:affine-linear-AC-vecsp}), hence the group homomorphism
$AC_{L^p}(f,G)$ is smooth.
\end{proof}

\begin{lemma}\label{lem:AC-Lie-gr-locality-axiom}
Let $G$ be a Lie group modelled on a sequentially complete locally convex space $E$, let $p\in[1,\infty]$.
Then the function
	\begin{align*}
		\Gamma_G\colon \ACp abG\to \prod_{j=1}^n \ACp {t_{j-1}}{t_j}G,\quad
		\eta\mapsto \left(\eta|_{[t_{j-1},t_j]}\right)_{j=1,\ldots,n}
	\end{align*}
is a smooth homomorphism and a smooth diffeomorphism onto a Lie subgroup of 
$\prod_{j=1}^n \ACp {t_{j-1}}{t_j}G$.
\end{lemma}

\begin{proof}
First of all we introduce some notations.
For $j=1,\ldots,n$ we denote $G_j:=\ACp {t_{j-1}}{t_j}G$, and for an open identity neighborhood
$U\subseteq G$ and a chart $\varphi\colon U\to V$ we write
$U_j:=\ACp {t_{j-1}}{t_j}U$, $V_j:=\ACp {t_{j-1}}{t_j}V$ and 
$\varphi_j\colon U_j\to V_j, \zeta\mapsto \varphi\circ\zeta$.

Clearly, the map $\Gamma_G$ is a group homomorphism and
	\begin{align*}
		\im(\Gamma_G)=\{(\eta_j)_{j=1,\ldots,n}\in\prod_{j=1}^n G_j :
		 \eta_{j-1}(t_j)=\eta_j(t_j) \mbox{ for all }j\in\{2,\ldots,n\}\}
	\end{align*}
is a subgroup of $\prod_{j=1}^n G_j$. Moreover, the function
	\begin{align*}
		\psi:=\prod_{j=1}^n \varphi_j \colon \prod_{j=1}^n U_j\to\prod_{j=1}^n V_j,\quad
		(\zeta_j,\ldots,\zeta_n)\mapsto (\varphi\circ\zeta_1,\ldots,\varphi\circ\zeta_n)
	\end{align*}
is a chart for $\prod_{j=1}^n G_j$ and 
$\psi(\im(\Gamma_G)\cap \prod_{j=1}^n U_j) = \im(\Gamma_E)\cap \prod_{j=1}^n V_j$,
where $\Gamma_E$ is the linear topological embedding with closed image from Lemma
\ref{lem:AC-vecsp-locality-axiom}. Therefore, $\im(\Gamma_G)$ is a Lie subgroup
modelled on the closed vector subspace $\im(\Gamma_E)$ of $\prod_{j=1}^n \ACp {t_{j-1}}{t_j}E$.

Finally, both compositions
	\begin{align*}
		\psi\circ\Gamma_G\circ\ACp ab\varphi^{-1}\colon \ACp abV\to\prod_{j=1}^n V_j,\quad
		\eta\mapsto \Gamma_E(\eta)
	\end{align*}
and
	\begin{align*}
		\ACp ab\varphi\circ \Gamma_G^{-1}\circ\left(\psi|_{\im(\Gamma_G)}\right)^{-1}\colon 
		\im(\Gamma_E)\cap \prod_{j=1}^n V_j&\to\ACp abV,\\
		\eta&\mapsto\Gamma_E^{-1}(\eta)
	\end{align*}
are smooth maps, thus we conclude that $\Gamma_G$ is a smooth diffeomorphism onto its image.
\end{proof}

See \cite[Lemma 5.10]{GlMR} for the next result.

\begin{proposition}\label{prop:double-push-forward-Lp-gr}
Let $G$ be a smooth Lie group, let $E$, $F$ be locally convex spaces and $f\colon G\times E\to F$
be a $C^{k+1}$-function and linear in the second argument. Then for $p\in[1,\infty]$ the function
	\begin{align}\label{eq:double-pushforward-Lp-gr-map}
		C([a,b],G)\times\Leb pabE\to\Leb pabF,\quad
		(\eta,[\gamma])\mapsto [f\circ(\eta,\gamma)]
	\end{align}
is $C^k$.
\end{proposition}

\begin{proof}
The function is defined by Lemma \ref{lem:linear2-in-Lp}. We fix some $\bar{\eta}\in C([a,b],G)$
and some open identity neighborhood $U\subseteq G$. Then $U$ contains some open identity
neighborhood $W$ such that $WW\subseteq U$.  The function in (\ref{eq:double-pushforward-Lp-gr-map})
will be $C^k$ if the restriction
	\begin{align}
		Q\times\Leb pabE\to\Leb pabF,\quad
		(\eta,[\gamma])\mapsto [f\circ(\eta,\gamma)]
	\end{align}
is $C^k$, where $Q:=\{\zeta\in C([a,b],G) : \bar{\eta}^{-1}\cdot\zeta\in C([a,b],W) \}$ is an
open neighborhood of $\bar{\eta}$.

Consider a partition $\abpart$ such that
	\begin{align*}
		\bar{\eta}(t_{j-1})^{-1}\bar{\eta}([t_{j-1},t_j])\subseteq W.
	\end{align*}
From Lemma \ref{lem:Lp-locality-axiom} follows that the above function will be $C^k$
if 
	\begin{align*}
		Q\times\Leb pabE&\to\prod_{j=1}^n \Leb p{t_{j-1}}{t_j}F,\\
		(\eta,[\gamma])&\mapsto \left([f\circ(\eta,\gamma)|_{[t_{j-1},t_j]}]\right)_{j=\{1,\ldots,n\}}
	\end{align*}
is $C^k$, which will be the case if each component
	\begin{align}\label{eq:double-pushforward-Lp-gr-step1}
		Q\times\Leb pabE\to\Leb p{t_{j-1}}{t_j}F,\quad
		(\eta,[\gamma])\mapsto [f\circ(\eta,\gamma)|_{[t_{j-1},t_j]}]
	\end{align}
is $C^k$.

Now, by \ref{lem:Lp-locality-axiom}, the function
	\begin{align*}
		Q\times\Leb pabE&\to C([t_{j-1},t_j],G)\times\Leb p{t_{j-1}}{t_j}E,\\
		(\eta,[\gamma])&\mapsto (\bar{\eta}(t_{j-1})^{-1}\eta|_{[t_{j-1},t_j]},[\gamma|_{[t_{j-1},t_j]}])
	\end{align*}
is smooth; for $\eta\in Q$ and $t\in[t_{j-1},t_j]$ we have 
	\begin{align*}
		\bar{\eta}(t_{j-1})^{-1}\eta(t)
		=\bar{\eta}(t_{j-1})^{-1}\bar{\eta}(t)\bar{\eta}(t)^{-1}\eta(t)
		\in WW\subseteq U.
	\end{align*}
Thus
	\begin{align*}
		Q\times\Leb pabE&\to C([t_{j-1},t_j],V)\times\Leb p{t_{j-1}}{t_j}E,\\
		(\eta,[\gamma])&\mapsto (\varphi\circ\bar{\eta}(t_{j-1})^{-1}\eta|_{[t_{j-1},t_j]},[\gamma|_{[t_{j-1},t_j]}])
	\end{align*}
is smooth, where $\varphi\colon U\to V$ is a chart for $G$. We define
	\begin{align*}
		g\colon V\times E\to F,\quad (x,y)\mapsto f(\bar{\eta}(t_{j-1})\varphi^{-1}(x),y),
	\end{align*}
which is $C^{k+1}$ and linear in the second argument, and we use
Proposition \ref{prop:double-push-forward-Lp} to obtain a $C^k$-map
	\begin{align*}
		Q\times\Leb pabE&\to\Leb p{t_{j-1}}{t_j}F,\\
		(\eta,[\gamma])&\mapsto 
		[g\circ(\varphi\circ\bar{\eta}(t_{j-1})^{-1}\eta|_{[t_{j-1},t_j]},\gamma|_{[t_{j-1},t_j]})]
		= [f\circ (\eta,\gamma)|_{[t_{j-1},t_j]}],
	\end{align*}
which is exactly the required function from (\ref{eq:double-pushforward-Lp-gr-step1}).
\end{proof}

\section{Measurable regularity of Lie groups}

\subsection{Logarithmic derivatives of \emph{AC}-functions}

The following definition is taken from \cite[Definition 5.1]{GlMR}.

\begin{definition}\label{def:derivative-AC-mfd}
Let $M$ be a smooth manifold modelled on a sequentially complete locally convex space $E$, 
let $p\in[1,\infty]$.
Consider $\eta\in\ACp abM$, a partition $\abpart$ and charts $\varphi_j\colon U_j\to V_j$
for $M$ such that $\eta([t_{j-1},t_j])\subseteq U_j$ for all $j$ and
	\begin{align*}
		\eta_j:=\varphi_j\circ\eta|_{[t_{j-1},t_j]}\in\ACp {t_{j-1}}{t_j}E.
	\end{align*}
Denote $\eta_j^\prime:=[\gamma_j]\in\Leb p{t_{j-1}}{t_j}E$ and set
	\begin{align*}
		\gamma(t):=T\varphi_j^{-1}(\eta_j(t),\gamma_j(t))
	\end{align*}
for $t\in[t_{j-1},t_j[$, and
	\begin{align*}
		\gamma(b):=T\varphi_n^{-1}(\eta_n(b),\gamma_n(b)).
	\end{align*}
The constructed function $\gamma\colon[a,b]\to TM$ is measurable 
and we write $\dot\eta:=[\gamma]$.
\end{definition}

\begin{remark}\label{rem:derivative-AC-mfd-all-charts}
Note that the definition of $\dot\eta$ is independent of the choice of the partition $\abpart$
and the charts $\varphi_j$.
\end{remark}

Recall that on the tangent bundle $TG$ of a smooth Lie group $G$, there is a smooth
Lie group structure with multiplication $Tm_G\colon TG\times TG\to TG$ and inversion
$Tj_G\colon TG\to TG$. For $g$, $h\in G$, $v\in T_gG$ and $w\in T_hG$ write
	\begin{align*}
		v.h:= T_g\rho_h(v)\in T_{gh}G,\quad g.w:=T_h\lambda_g(w)\in T_{gh}G,
	\end{align*}
(where $\lambda_g\colon x\mapsto gx$, $\rho_h\colon x\mapsto xh$ are the left and the right
translations on $G$)
then
	\begin{align*}
		Tm_G(v,w)=g.w+v.h.
	\end{align*}

See \cite[Lemma 5.4]{GlMR} for the following computations.

\begin{lemma}\label{lem:AC-Lie-gr-derivative-rules}
Let $G$ be a Lie group modelled on a sequentially complete locally convex space, let $p\in[1,\infty]$.
For $\eta$, $\zeta\in\ACp abG$ with $\dot{\eta} = [\gamma]$, $\dot{\zeta} = [\xi]$
we have
	\begin{align}\label{eq:AC-Lie-gr-product-rule}
		(\eta\cdot\zeta)^{\cdot}= [t\mapsto \gamma(t).\zeta(t) + \eta(t).\xi(t)]
	\end{align}
and
	\begin{align}\label{eq:AC-Lie-gr-quotient-rule}
		(\eta^{-1})^{\cdot} = [t\mapsto -\eta(t)^{-1}.\gamma(t).\eta(t)^{-1}].
	\end{align}
Further, if $f\colon G\to H$ is a smooth function between Lie groups modelled on sequentially complete
locally convex spaces, then
	\begin{align}\label{eq:AC-Lie-gr-chain-rule}
		(f\circ\eta)^{\cdot} = [Tf\circ\gamma].
	\end{align}
\end{lemma}

\begin{proof}
We prove the last equation (\ref{eq:AC-Lie-gr-chain-rule}) first. 
Consider a partition $\abpart$ and charts $\varphi_j\colon U_j\to V_j$,
$\psi\colon P_j\to Q_j$ for $G$ and $H$, respectively, such that
	\begin{align*}
		&\varphi_j\circ\eta|_{[t_{j-1},t_j]}\in\ACp {t_{j-1}}{t_j}E,\\
		&\psi_j\circ f\circ\eta|_{[t_{j-1},t_j]}\in\ACp {t_{j-1}}{t_j}F,
	\end{align*}	 
where $E$ and $F$ are the model spaces of $G$ and $H$. Denote
	\begin{align*}
		&[\gamma_j] := (\varphi_j\circ\eta|_{[t_{j-1},t_j]})^\prime\in\Leb p{t_{j-1}}{t_j}E, \\
		&[\xi_j] := (\psi_j\circ f\circ\eta|_{[t_{j-1},t_j]})^\prime\in\Leb p{t_{j-1}}{t_j}F.
	\end{align*}
Then (using (\ref{eq:AC-vecsp-chain-rule})) we have
	\begin{align*}
		[\xi_j] = ((\psi_j\circ f \circ\varphi_j^{-1})\circ(\varphi_j\circ\eta|_{[t_{j-1},t_j]}))^\prime
		= [d(\psi_j\circ f \circ\varphi_j^{-1})(\varphi_j\circ\eta|_{[t_{j-1},t_j]},\gamma_j)].
	\end{align*}
Therefore, for $[\delta] := (f\circ\eta)^{\cdot}$ and $t\in[t_{j-1},t_j[$ we have
	\begin{align*}
		\delta(t) 
		&= T\psi_j^{-1} ((\psi_j\circ f\circ\eta)(t),
		d(\psi_j\circ f \circ\varphi_j^{-1})((\varphi_j\circ\eta)(t),\gamma_j(t)))\\
		& = T\psi_j^{-1}((\psi_j\circ f\circ\varphi_j^{-1}\circ\varphi_j\circ\eta)(t),
		d(\psi_j\circ f \circ\varphi_j^{-1})((\varphi_j\circ\eta)(t),\gamma_j(t)))\\
		& = (T\psi_j^{-1}\circ T(\psi_j\circ f\circ\varphi_j^{-1}))((\varphi_j\circ\eta)(t),\gamma_j(t))\\
		&=(Tf\circ T\varphi_j^{-1})((\varphi_j\circ\eta)(t),\gamma_j(t))\\
		&= (Tf\circ\gamma)(t),
	\end{align*}
and analogously for $t=b$.

Now,
	\begin{align*}
		(\eta\cdot\zeta)^{\cdot}
		=(m_G\circ (\eta,\zeta))^{\cdot}
		=[Tm_G\circ(\gamma,\xi)]
		=[t\mapsto \eta(t).\xi(t) + \gamma(t).\zeta(t)]
	\end{align*}
and
	\begin{align*}
		(\eta^{-1})^{\cdot}
		=(j_G\circ\eta)^{\cdot}
		=[Tj_G\circ\gamma]
		=[t\mapsto -\eta(t)^{-1}.\gamma(t).\eta(t)^{-1}].
	\end{align*}
\end{proof}

The logarithmic derivative of $\eta\in\ACp abG$ is defined as follows (see \cite[Definition 5.6]{GlMR}).

\begin{definition}\label{def:log-derivative}
For $\eta\in\ACp abG$ define the \emph{left logarithmic derivative of $\eta$} via
	\begin{align*}
		\delta(\eta):=[\omega_l\circ\gamma],
	\end{align*}
where $[\gamma]=\dot\eta$ and $\omega_l\colon TG\to\mathfrak{g}, v\mapsto \pi_{TG}(v)^{-1}. v$
with the bundle projection $\pi_{TG}\colon TG\to G$.
\end{definition}

As in \cite[Lemma 5.11]{GlMR}, we prove:

\begin{lemma}\label{lem:log-derivative-in-Lp}
Let $G$ be a Lie group modelled on a sequentially complete locally convex space $E$
and $p\in[1,\infty]$. If $\eta\in\ACp abG$,
then $\delta(\eta)\in\Leb pab{\mathfrak{g}}$.
\end{lemma}

\begin{proof}
By definition, there exists a partition $\abpart$ and there exist charts $\varphi_j\colon U_j\to V_j$
for $G$ such that $\eta([t_{j-1},t_j])\subseteq U_j$ and 
$\eta_j:=\varphi_j\circ\eta|_{[t_{j-1},t_j]}\in\ACp {t_{j-1}}{t_j}E$ for every $j\in\{1,\ldots,n\}$.
We denote $[\gamma_j]:=\eta_j^\prime$ and $[\gamma]:=\dot{\eta}$ and see that
	\begin{align*}
		\omega_l\circ\gamma|_{[t_{j-1},t_j]}
		= \omega_l\circ T\varphi_j^{-1}\circ(\eta_j,\gamma_j)\in\LLeb p{t_{j-1}}{t_j}{\mathfrak{g}}
	\end{align*}
by Lemma \ref{lem:linear2-in-Lp}, since $\omega_l\circ T\varphi_j^{-1}\colon V_j\times E\to\mathfrak{g}$
is continuous and linear in the second argument. From Lemma \ref{lem:Lp-locality-axiom}
follows $\delta(\eta) = [\omega_l\circ\gamma]\in\Leb pab{\mathfrak{g}}$.
\end{proof}

Some of the following properties can be found in \cite[Lemma 5.12]{GlMR}.

\begin{lemma}\label{lem:log-derivative-rules}
Let $\eta$, $\zeta\in\ACp abG$ and denote $\delta(\eta)=[\gamma]$, $\dot{\eta}=[\bar{\gamma}]$,
$\delta(\zeta)=[\xi]$, $\dot{\zeta}=[\bar{\xi}]$. Then the following holds.
\begin{itemize}
\item[(i)] We have
	\begin{align}\label{eq:log-derivative-product-rule}
		\delta(\eta\cdot\zeta) = [t\mapsto \zeta(t)^{-1}.\gamma(t).\zeta(t) + \xi(t)],
	\end{align}
and
	\begin{align}\label{eq:log-derivative-quotient-rule}
		\delta(\eta^{-1}) = [t\mapsto -\bar{\gamma}(t).\eta(t)^{-1}].
	\end{align}
\item[(ii)]
We have $\delta(\eta)=0$ if and only if $\eta$ is constant.
\item[(iii)]
We have $\delta(\eta)=\delta(\zeta)$ if and only if $\eta=g\zeta$ for some $g\in G$.
\item[(iv)]
If $f\colon G\to H$ is a Lie group homomorphism, then
	\begin{align}\label{eq:log-derivative-chain-rule}
		\delta(f\circ\eta) = [L(f)\circ\gamma].
	\end{align}
\item[(v)]
For $a\leq\alpha<\beta\leq b$ and $f\colon[c,d]\to[a,b], f(t):=\alpha + \frac{t-c}{d-c}(\beta-\alpha)$
we have
	\begin{align}\label{eq:log-derivative-affine-linear}
		\delta(\eta\circ f) = \frac{\beta-\alpha}{d-c}[\gamma\circ f].
	\end{align}
\end{itemize}
\end{lemma}

\begin{proof}
(i)
Using Equations (\ref{eq:AC-Lie-gr-product-rule}) and (\ref{eq:AC-Lie-gr-quotient-rule}), we get
	\begin{align*}
		\delta(\eta\cdot\zeta)
		&= [t\mapsto (\eta(t)\zeta(t))^{-1}.(\bar{\gamma}(t).\zeta(t) + \eta(t).\bar{\xi}(t))]\\
		&= [t\mapsto (\zeta(t)^{-1}\eta(t)^{-1}).\bar{\gamma}(t).\zeta(t) + (\zeta(t)^{-1}\eta(t)^{-1}).\eta(t).\bar{\xi}(t)]\\
		&= [t\mapsto \zeta(t)^{-1}.\gamma(t).\zeta(t) + \xi(t)],
	\end{align*}
and
	\begin{align*}
		\delta(\eta^{-1})
		= [t\mapsto \eta(t).(-\eta(t)^{-1}.\bar{\gamma}(t).\eta(t)^{-1})]
		= [t\mapsto -\bar{\gamma}(t).\eta(t)^{-1}].
	\end{align*}

(ii)
Now, we assume that $\delta(\eta)=0$, that is, 
$[t\mapsto \eta(t)^{-1}.\bar{\gamma}(t)] = 0\in\Leb pab{\mathfrak{g}}$. In other words, 
$\eta(t)^{-1}.\bar{\gamma}(t)=0\in\mathfrak{g}$ for a.e. $t\in[a,b]$. 
Let $\abpart$, charts $\varphi_j$ and $[\gamma_j]$ be as in Definition \ref{def:abs-cont-mfd}.
Then for $\bar{\gamma}(t)\in T_{\eta(t)}G$ we have $d\varphi_j(\bar{\gamma}(t)) = 0\in E$
for a.e. $t\in[t_{j-1},t_j]$. On the other hand, we have $d\varphi_j(\bar{\gamma}(t)) = \gamma_j(t)$
for a.e. $t\in[t_{j-1},t_j]$, thus $[\gamma_j]=0 \in\Leb p{t_{j-1}}{t_j}E$. That means, that 
$\varphi\circ\eta|_{[t_{j-1},t_j]}$ is constant, whence $\eta|_{[t_{j-1},t_j]}$ is constant,
whence $\eta$ is constant, being continuous.

Conversely, assume $\eta(t)=g\in G$ for all $t\in[a,b]$. Then for some chart $\varphi$ around $g$
we have
	\begin{align*}
		\varphi(g) = \varphi(\eta(t)) = \varphi(g) + \int_a^t \gamma_g(s) \,ds
	\end{align*}
for every $t\in[a,b]$, thus $\gamma_g(s)=0$ for a.e. $s\in[a,b]$, by Lemma \ref{lem:weak-integral-unique},
in other words, $(\varphi\circ\eta)^\prime = 0\in\Leb pabE$. 
Therefore,
	\begin{align*}
		\bar{\gamma}(t) = T\varphi^{-1}(\varphi(\eta(t)), 0) = T\varphi^{-1}(\varphi(g),0)
	\end{align*}
a.e.,
whence
	\begin{align*}
		\delta(\eta) = [t\mapsto \eta(t)^{-1}.T\varphi^{-1}(\varphi(g),0)] = [t\mapsto g^{-1}.T\varphi^{-1}(\varphi(g),0)]
		=0 \in\Leb pab{\mathfrak{g}}.
	\end{align*}

(iii)
Now, assume $[\gamma] = \delta(\eta) = \delta(\zeta) = [\xi]$, then (using Equations (\ref{eq:log-derivative-product-rule})
and (\ref{eq:log-derivative-quotient-rule}))
	\begin{align*}
		\delta(\eta\cdot\zeta^{-1})
		&= [t\mapsto \zeta(t).\gamma(t).\zeta(t)^{-1} - \bar{\xi}(t).\zeta(t)^{-1}]\\
		&= [t\mapsto \zeta(t).\xi(t).\zeta(t)^{-1} - \bar{\xi}(t).\zeta(t)^{-1}]\\
		&= [t\mapsto \bar{\xi}(t).\zeta(t)^{-1} - \bar{\xi}(t).\zeta(t)^{-1}]
		=0 \in\Leb pab{\mathfrak{g}}.
	\end{align*}
Then, by the above, the curve $\eta\cdot\zeta^{-1}$ is constant, say $\eta\cdot\zeta^{-1} = g\in G$, thus $\eta = g\zeta$.

Conversely, assume $\eta = g\zeta$. We define $\eta_g\colon [a,b]\to G, t\mapsto g$ in $\ACp abG$,
then  $[\gamma_g]=\delta(\eta_g)=0\in\Leb pab{\mathfrak{g}}$ (by the above), whence
	\begin{align*}
		\delta(\eta) = \delta(\eta_g\cdot\zeta)
		= [t\mapsto \xi(t)] = \delta(\zeta),
	\end{align*}
using Equation (\ref{eq:log-derivative-product-rule}).

(iv)
By (\ref{eq:AC-Lie-gr-chain-rule}), we have
	\begin{align*}
		\delta(f\circ\eta)
		&= [t\mapsto f(\eta(t))^{-1}.Tf(\bar{\gamma}(t))]\\
		&= [t\mapsto (T\lambda_{f(\eta(t))^{-1}}^H\circ Tf)(\bar{\gamma}(t))]\\
		&= [t\mapsto T(\lambda_{f(\eta(t))^{-1}}^H\circ f)(\bar{\gamma}(t))]\\
		&= [t\mapsto T(f\circ\lambda_{\eta(t)^{-1}}^G)(\bar{\gamma}(t))]\\
		&= [t\mapsto Tf(\eta(t)^{-1}.\bar{\gamma}(t))]\\
		&= [t\mapsto L(f)(\gamma(t))]
		= [L(f)\circ\gamma].
	\end{align*}
Note that we used $\lambda_{f(\eta(t))^{-1}}^H\circ f = f\circ\lambda_{\eta(t)^{-1}}^G$ as $f$ is a group homomorphism.

(v) 
Consider a partition $c=t_0<t_1<\ldots<t_n=d$ and charts $\varphi\colon U_j\to V_j$
with $\eta(f([t_{j-1},t_j]))\subseteq U_j$.
Write
	\begin{align*}
		f_j:= f|_{[t_{j-1},t_j]},\quad \eta_j := \eta|_{[f(t_{j-1}),f(t_j)]}.
	\end{align*}
Identifying equivalence classes with functions, we obtain
	\begin{align*}
		\delta(\eta\circ f)|_{[t_{j-1},t_j]}
		&=\omega_l \circ T\varphi_j^{-1}\circ
		(\varphi_j\circ\eta\circ f_j, (\varphi_j\circ\eta\circ f_j)^\prime)\\
		&=\omega_l \circ T\varphi_j^{-1}\circ
		(\varphi_j\circ\eta\circ f_j, 
		\frac{\beta-\alpha}{d-c}(\varphi_j\circ\eta_j)^\prime\circ f_j)\\
		&=\omega_l \circ T\varphi_j^{-1}\circ
		(\varphi\circ\eta_j,\frac{\beta-\alpha}{d-c}(\varphi_j\circ\eta_j)^\prime)\circ f_j\\
		&=\frac{\beta-\alpha}{d-c}\left(\omega_l \circ T\varphi_j^{-1}\circ
		(\varphi\circ\eta_j,(\varphi_j\circ\eta_j)^\prime)\circ f_j\right)\\
		&=\frac{\beta-\alpha}{d-c}(\delta(\eta)\circ f|_{[t_{j-1},t_j]}),
	\end{align*}
using the formula in Lemma \ref{lem:affine-linear-AC-vecsp} and the linearity of
$\omega_l\circ T\varphi_j^{-1}$ in its second argument.
\end{proof}

\subsection{$\mathbf{L^p}$-regularity}

The $L^p$-semiregularity and $L^p$-regularity of Lie groups modelled on sequentially
complete locally convex spaces is defined as in \cite[Definition 5.16]{GlMR}.

\begin{definition}\label{def:Lp-regular-Lie-gr}
Let $G$ be a  smooth Lie group modelled on a sequentially complete locally convex space. 
For $p\in[1,\infty]$, the Lie group $G$ is called
\emph{$L^p$-semiregular} if for every $\gamma\in\Leb p01{\mathfrak{g}}$ the initial value
problem
	\begin{align}\label{eq:initial-value-problem}
		\delta(\eta)=\gamma,\quad \eta(0)=e
	\end{align}
has a solution $\eta_\gamma\in\ACp abG$ (which is unique, by Lemma \ref{lem:log-derivative-rules}).

An $L^p$-semiregular Lie group $G$ is called \emph{$L^p$-regular} if the obtained function
	\begin{align}
		\Evol\colon\Leb p01{\mathfrak{g}} \to \ACp abG,\quad
		\gamma\mapsto\eta_\gamma
	\end{align}
is smooth.
\end{definition}

\begin{remark}\label{rem:regular-evol-smooth}
As in \cite[Remark 5.18]{GlMR}, we note that
if a Lie group $G$ is $L^p$-regular, then the function
	\begin{align*}
		\evol\colon \Leb p01{\mathfrak{g}}\to G,\quad \gamma\mapsto \Evol(\gamma)(1)
	\end{align*}
is smooth, since so is 
the evaluation map $\ev_1\colon \ACp 01G\to G, \eta\mapsto \eta(1)$
(see Lemma \ref{lem:eval-smooth-AC-Liegr}).
\end{remark}

We prove the result from \cite[Proposition 5.20]{GlMR}.

\begin{proposition}\label{prop:Evol-smoothness}
Let $G$ be an $L^p$-semiregular Lie group. Then the function 
$\Evol$ is smooth if and only if $\Evol$
is smooth as a function to $C([0,1],G)$.
\end{proposition}

\begin{proof}
First assume that $\Evol\colon\Leb p01{\mathfrak{g}} \to \ACp 01G$ is smooth.
As the inclusion map $\incl\colon \ACp 01G\to C([0,1],G)$ is smooth (see Lemma \ref{lem:incl-AC-C-Lie-gr}),
the composition $\incl\circ\Evol\colon \Leb p01{\mathfrak{g}}\to C([0,1],G)$ is smooth.

Conversely, assume that $\Evol\colon \Leb p01{\mathfrak{g}}\to C([0,1],G)$ is smooth; for
some fixed $\bar{\gamma}\in\Leb p01{\mathfrak{g}}$ we are going to find some open neighborhood $P$
of $\bar{\gamma}$ such that $\Evol|_P\colon P\to\ACp 01G$ is smooth.

To this end, let $U\subseteq G$ be an open identity neighborhood and $\varphi\colon U\to V$
be a chart. Then $U$ contains some
open identity neighborhood $W$ such that $WW\subseteq U$. For 
$\eta_{\bar{\gamma}}:=\Evol(\bar{\gamma})$, the subset
	\begin{align*}
		Q := \{\zeta\in C([0,1],G) : \eta_{\bar{\gamma}}^{-1}\cdot\zeta\in C([0,1],W) \} 
	\end{align*}
is an open neighborhood of $\eta_{\bar{\gamma}}$. Set
	\begin{align*}
		P := \Evol^{-1}(Q).
	\end{align*}
Now, we want to show that the function
	\begin{align}\label{eq:Evol-step1}
		P\to\ACp 01G, \quad \gamma\mapsto \eta_\gamma:=\Evol(\gamma)
	\end{align}
is smooth.

As $\eta_{\bar{\gamma}}$ is continuous, there exists a partition $0=t_0<t_1<\cdots<t_n=1$ 
such that 
$\eta_{\bar{\gamma}}(t_j)^{-1}\eta_{\bar{\gamma}}([t_{j-1},t_j])\subseteq W$ for each $j\in\{1,\ldots,n\}$.
Using the function $\Gamma_G$ from Lemma \ref{lem:AC-Lie-gr-locality-axiom}, the
map in (\ref{eq:Evol-step1}) will be smooth if
	\begin{align}\label{eq:Evol-step2}
		P\mapsto \prod_{j=1}^n \ACp {t_{j-1}}{t_j}G,\quad 
		\gamma\mapsto \left(\eta_{\gamma}|_{[t_{j-1},t_j]}\right)_{j=1,\ldots,n}
	\end{align}
is smooth, which will be the case if each of the components
	\begin{align}\label{eq:Evol-step3}
		P\mapsto \ACp {t_{j-1}}{t_j}G,\quad
		\gamma\mapsto \eta_{\gamma}|_{[t_{j-1},t_j]}
	\end{align}
is smooth. As left translations on the Lie group $\ACp {t_{j-1}}{t_j}G$ are smooth
diffeomorphisms, the function in (\ref{eq:Evol-step3}) will be smooth if
	\begin{align}\label{eq:Evol-step4}
		P\mapsto \ACp {t_{j-1}}{t_j}G,\quad
		\gamma\mapsto \eta_{\bar{\gamma}}(t_j)^{-1}\eta_{\gamma}|_{[t_{j-1},t_j]}
	\end{align}
is a smooth map.

Now, for every $t\in[t_{j-1},t_j]$ we have
	\begin{align*}
		\eta_{\bar{\gamma}}(t_j)^{-1}\eta_{\gamma}(t) = 
		\eta_{\bar{\gamma}}(t_j)^{-1}\eta_{\bar{\gamma}}(t)\eta_{\bar{\gamma}}(t)^{-1}\eta_{\gamma}(t)
		\in WW\subseteq U,
	\end{align*}
in other words, $\eta_{\bar{\gamma}}(t_j)^{-1}\eta_{\gamma}|_{[t_{j-1},t_j]}\in\ACp {t_{j-1}}{t_j}U$.
Thus the smoothness of (\ref{eq:Evol-step4}) will follow from the smoothness of
	\begin{align}\label{eq:Evol-step5}
		P\to \ACp {t_{j-1}}{t_j}E,\quad
		\gamma\mapsto \varphi\circ\eta_{\bar{\gamma}}(t_j)^{-1}\eta_{\gamma}|_{[t_{j-1},t_j]}.
	\end{align}
Using the definition of the
topology on $\ACp {t_{j-1}}{t_j}E$ (see Definition \ref{def:abs-cont-vector-space}), we will show that
	\begin{align*}
		P\to E\times \Leb p{t_{j-1}}{t_j}E,\quad
		\gamma\mapsto (\varphi(\eta_{\bar{\gamma}}(t_j)^{-1}\eta_{\gamma}(t_{j-1})),
		(\varphi\circ\eta_{\bar{\gamma}}(t_j)^{-1}\eta_{\gamma}|_{[t_{j-1},t_j]})^\prime)
	\end{align*}	
is smooth.

Using the assumed smoothness of $P\to C([0,1],G), \gamma\mapsto\eta_{\gamma}$, we
see that the
first component of the above function is smooth.
Therefore, it remains to show that 
	\begin{align}\label{eq:Evol-step7}
		P\to\Leb p{t_{j-1}}{t_j}E,\quad
		\gamma\mapsto (\varphi\circ\eta_{\bar{\gamma}}(t_j)^{-1}\eta_{\gamma}|_{[t_{j-1},t_j]})^\prime
	\end{align}
is smooth.

Identifying equivalence classes with functions, we have
	\begin{align*}
		(\varphi\circ\eta_{\bar{\gamma}}(t_j)^{-1}\eta_{\gamma}|_{[t_{j-1},t_j]})^\prime
		= d\varphi\circ (\eta_{\bar{\gamma}}(t_j)^{-1}\eta_{\gamma}|_{[t_{j-1},t_j]})^{\cdot}.
	\end{align*}
Consider the smooth function
	\begin{align*}
		\sigma\colon G\times\mathfrak{g}\to TG,\quad (g,v)\mapsto g.v.
	\end{align*}
We have 
	\begin{align*}
		d\varphi\circ (\eta_{\bar{\gamma}}(t_j)^{-1}&\eta_{\gamma}|_{[t_{j-1},t_j]})^{\cdot}\\
		&= d\varphi\circ\sigma\circ
		(\eta_{\bar{\gamma}}(t_j)^{-1}\eta_{\gamma}|_{[t_{j-1},t_j]},
		\delta(\eta_{\bar{\gamma}}(t_j)^{-1}\eta_{\gamma}|_{[t_{j-1},t_j]}))\\
		&=d\varphi\circ\sigma\circ
		(\eta_{\bar{\gamma}}(t_j)^{-1}\eta_{\gamma}|_{[t_{j-1},t_j]},
		\delta(\eta_{\gamma}|_{[t_{j-1},t_j]}))\\
		&=d\varphi\circ\sigma\circ
		(\eta_{\bar{\gamma}}(t_j)^{-1}\eta_{\gamma}|_{[t_{j-1},t_j]},\gamma|_{[t_{j-1},t_j]}),
	\end{align*}
using $(iii)$ from Lemma \ref{lem:log-derivative-rules}. Hence the map in (\ref{eq:Evol-step7})
will be smooth if
	\begin{align}
		P\to\Leb p{t_{j-1}}{t_j}E,\quad
		\gamma\mapsto d\varphi\circ\sigma\circ
		(\eta_{\bar{\gamma}}(t_j)^{-1}\eta_{\gamma}|_{[t_{j-1},t_j]},\gamma|_{[t_{j-1},t_j]})
	\end{align}
is smooth. But this is true, the function being a composition of the smooth functions
	\begin{align*}
		P\to C([t_{j-1},t_j],U)\times \Leb p{t_{j-1}}{t_j}E,\quad 
		\gamma\mapsto (\eta_{\bar{\gamma}}(t_j)^{-1}\eta_{\gamma}|_{[t_{j-1},t_j]},\gamma|_{[t_{j-1},t_j]})
	\end{align*}
and
	\begin{align*}
		C([t_{j-1},t_j],U)\times\Leb p{t_{j-1}}{t_j}{\mathfrak{g}}\to\Leb p{t_{j-1}}{t_j}{\mathfrak{g}}&,\quad
		(\eta,\gamma)\mapsto d\varphi\circ\sigma\circ(\eta,\gamma),
	\end{align*}
(the smoothness of the last function holds by Proposition \ref{prop:double-push-forward-Lp-gr}, as
the composition 
$d\varphi\circ\sigma\colon G\times\mathfrak{g}\to E$ is linear in the second argument).
\end{proof}

Analogously to \cite[Corollary 5.21]{GlMR} we get the following:

\begin{corollary}\label{cor:Lie-gr-Lp-then-Lq-regular}
Let $G$ be a Lie group and $p,q\in[1,\infty]$ with $q\geq p$. If
$G$ is $L^p$-regular, then $G$ is $L^q$-regular. Furthermore, in this case $G$ is $C^0$-regular.
\end{corollary}

\begin{proof}
Assume that $G$ is $L^p$-regular and $q\geq p$. Since 
$\Leb q01{\mathfrak{g}}\subseteq \Leb p01{\mathfrak{g}}$
with a smooth inclusion map (Remark \ref{rem:Lp-inclusions}),
the Lie group $G$ is $L^q$-semiregular
and $\Leb q01{\mathfrak{g}}\to C([0,1],G), \gamma\mapsto \Evol(\gamma)$ is smooth.
From Proposition \ref{prop:Evol-smoothness} follows, that $\Leb q01{\mathfrak{g}}\to AC_{L^q}([0,1],G),
\gamma\mapsto \Evol(\gamma)$ is smooth, whence $G$ is $L^q$-regular.

Further, since $C([0,1],\mathfrak{g})\subseteq \Leb p01{\mathfrak{g}}$, the Lie group is
$C^0$-semiregular. Since the inclusion map $\incl\colon C([0,1],\mathfrak{g})\to \Leb p01{\mathfrak{g}}$
is smooth, as well as the evaluation map $\ev_1\colon C([0,1],G)\to G$, the composition
$C([0,1],\mathfrak{g})\to G, \gamma\mapsto \Evol(\gamma)(1)$
is smooth, whence $G$ is $C^0$-regular.
\end{proof}

The next Proposition shows that it suffices for a Lie group $G$ to be $L^p$-regular,
if it is merely \emph{locally $L^p$-regular} (see \cite[Definition 5.19, Proposition 5.25]{GlMR}).

\begin{proposition}\label{prop:Lie-gr-locally-iff-regular}
Let $G$ be a Lie group modelled on a sequentially complete locally convex
space $E$, let $\mathfrak{g}$ denote the Lie algebra of $G$. 
Let $\Omega\subseteq \Leb p01{\mathfrak{g}}$ be an open $0$-neighbourhood.
If for every $\gamma\in\Omega$ the initial value problem (\ref{eq:initial-value-problem})
has a (necessarily unique) solution $\eta_{\gamma}\in\ACp 01G$, then $G$ is $L^p$-semiregular.
If, in addition, the function $\Evol\colon\Omega\to\ACp 01G, \gamma\mapsto \eta_{\gamma}$
is smooth, then $G$ is $L^p$-regular.
\end{proposition}

\begin{proof}
First, fix some $\gamma\in\Leb p01{\mathfrak{g}}$ and for 
$n\in\N$, $k\in\{0,\ldots,n-1\}$ define $\gamma_{n,k}\in\Leb p01{\mathfrak{g}}$
as in (\ref{eq:subdivision-function}). Let $Q$ be a continuous seminorm
on $\Leb p01{\mathfrak{g}}$ such that $B_1^{Q}(0)\subseteq\Omega$. By Lemma
\ref{lem:Lp-subdivision}, there exists some $n\in\N$ such that
$\gamma_{n,k}\in\Omega$ for $k\in\{0,\ldots,n-1\}$. We set 
$\eta_{n,k}:=\Evol(\gamma_{n,k})\in\ACp 01G$ and define
$\eta_{\gamma}\colon[0,1]\to G$ via
	\begin{align}\label{eq:Evol-constr-part1}
		\eta_{\gamma}(t):=(\eta_{n,0}\circ f_{n,0})(t),\quad \mbox{for } t\in[0,\nicefrac{1}{n}],
	\end{align}
and
	\begin{align}\label{eq:Evol-constr-part2}
	\eta_{\gamma}(t):=\eta_{n,0}(1)\cdots\eta_{n,k-1}(1)(\eta_{n,k}\circ f_{n,k})(t),\quad
	\mbox{for } t\in[\nicefrac{k}{n},\nicefrac{k+1}{n}],
	\end{align}
where
	\begin{align*}
		f_{n,k}\colon[\nicefrac{k}{n},\nicefrac{k+1}{n}]\to[0,1],\quad
		f_{n,k}(t):=nt-k.
	\end{align*}
Then we easily verify that the function $\eta_{\gamma}$ is continuous and from
Lemma \ref{lem:affine-linear-AC-liegr} follows that 
$\eta_{\gamma}|_{[\nicefrac{k}{n},\nicefrac{k+1}{n}]}
\in\ACp {\nicefrac{k}{n}}{\nicefrac{k+1}{n}}G$, whence $\eta_{\gamma}\in\ACp 01G$.
Furthermore, $\eta_{\gamma}(0)=e$ and $\delta(\eta_{\gamma}) = \gamma$,
see Lemma \ref{lem:log-derivative-rules}(v). Consequently, $\Evol(\gamma):=\eta_{\gamma}$ solves
the initial value problem in (\ref{eq:initial-value-problem}) for $\gamma$, whence
$G$ is $L^p$-semiregular.

Now, assume that $\Evol\colon\Omega\to\ACp 01G$ is smooth; we will show the smoothness
of $\Evol$ on some open neighborhood of $\gamma$. 
From the
continuity of each
	\begin{align*}
		\pi_{n,k}\colon\Leb p01{\mathfrak{g}}\to\Leb p01{\mathfrak{g}},\quad
		\xi\mapsto \xi_{n,k},
	\end{align*}
(see Lemma \ref{lem:affine-linear-Lp}) follows that there exists an open
neighborhood $W\subseteq\Leb p01{\mathfrak{g}}$ of $\gamma$ such that
$\pi_{n,k}(W)\subseteq \Omega$ for every $k\in\{0,\ldots,n-1\}$. Then
	\begin{align*}
		\Evol\colon W\to\ACp 01G,\quad \xi\mapsto\eta_{\xi}
	\end{align*}
is defined, where $\eta_{\xi}$ is as in (\ref{eq:Evol-constr-part1}) and 
(\ref{eq:Evol-constr-part2}). It will be smooth if we show 
(using Lemma \ref{lem:AC-Lie-gr-locality-axiom}) that each
	\begin{align}\label{eq:Evol-locally}
		W\to \ACp {\nicefrac{k}{n}}{\nicefrac{k+1}{n}}G,\quad 
		\xi\mapsto \eta_{\xi}|_{\left[\nicefrac{k}{n},\nicefrac{k+1}{n}\right]}
	\end{align}
is smooth.
But, by construction, we have
	\begin{align*}
		\eta_{\xi}|_{[0,\nicefrac{1}{n}]} = \Evol(\xi_{n,0})\circ f_{n,0}
	\end{align*}
and
	\begin{align*}
		\eta_{\xi}|_{[\nicefrac{k}{n},\nicefrac{k+1}{n}]}
		= \evol(\xi_{n,0})\cdots\evol(\xi_{n,k-1})\Evol(\xi_{n,k})\circ f_{n,k},
	\end{align*}
so the smoothness of (\ref{eq:Evol-locally}) follows from Lemma \ref{lem:affine-linear-AC-liegr}
and Remark \ref{rem:regular-evol-smooth}.
\end{proof}

\end{document}